\documentclass[12pt, a4paper]{amsart}

\usepackage[hmargin=30mm, vmargin=25mm, includefoot, twoside]{geometry}
\usepackage[bookmarksopen=true]{hyperref}

\usepackage{amsfonts,amssymb,verbatim}
\usepackage{latexsym}
\usepackage{mathrsfs}
\usepackage{stmaryrd}
\usepackage{xspace}
\usepackage{enumerate, paralist}
\usepackage{graphicx}
\usepackage[all]{xy}
\usepackage{extarrows}

\usepackage[usenames,dvipsnames]{color}

\usepackage{txfonts, pxfonts}

\usepackage{amsthm}
\usepackage{amsmath}

\newtheorem{thm}{Theorem}[section]
 \newtheorem{cor}[thm]{Corollary}
 \newtheorem{lem}[thm]{Lemma}
 \newtheorem{prop}[thm]{Proposition}

\newtheorem{introthm}{Theorem}

\newtheorem{introcor}[introthm]{Corollary}
\numberwithin{equation}{section}

 \theoremstyle{definition}
  \newtheorem{defn}[thm]{Definition}
  \newtheorem{question}[thm]{Question}

  \newtheorem{problem}[thm]{Problem}
 \theoremstyle{remark}
 \newtheorem{rem}[thm]{Remark}

\newtheorem*{claim*}{Claim}
\newcommand{\cstr}{\mathrm{C}^*}

\def\B{\mathfrak B}
\def\K{\mathfrak K}

\def\H{\mathcal H}
\def\G{\mathcal G}

\def\supp{\mathrm{supp}}
\def\diam{\mathrm{diam}}
\def\ppg{\mathrm{ppg}}

\def\N{\mathbb N}
\def\C{\mathbb C}

\def\Nd{\mathcal N}

\begin{document}

\title{Measured asymptotic expanders and Rigidity for Roe algebras}

\author{Kang Li, J\'{a}n \v{S}pakula, and Jiawen Zhang}

\address[K. Li]{Department of Mathematics, Friedrich-Alexander-University Erlangen–Nuremberg, Cauerstrasse 11, 91058 Erlangen, Germany.}
\email{kang.li@fau.de}

\address[J. \v{S}pakula]{School of Mathematics, University of Southampton, Highfield, SO17 1BJ, United Kingdom.}
\email{jan.spakula@soton.ac.uk}

\address[J. Zhang]{School of Mathematical Sciences, Fudan University, 220 Handan Road, Shanghai 200433, P. R. China.}
\email{jiawenzhang@fudan.edu.cn}

\date{}

\thanks{KL was supported by the Internal KU Leuven BOF project C14/19/088. KL has also received funding from the European Research Council (ERC) under the European Union's Horizon 2020 research and innovation programme (grant agreement no. 677120-INDEX)}
\thanks{J\v{S} was partially supported by Marie Curie FP7-PEOPLE-2013-CIG Coarse Analysis (631945), and by EPSRC EP/V002899/1. For the purpose of open access, the author has applied a Creative Commons Attribution (CC BY) licence to any Author Accepted Manuscript version arising.}
\thanks{JZ was supported by NSFC11871342}

% orig from Jiawen

\begin{abstract}
In this paper, we give a new geometric condition in terms of measured asymptotic expanders to ensure rigidity of Roe algebras.
%if $X$ and $Y$ are metric spaces with bounded geometry such that their Roe algebras are $*$-isomorphic, then $X$ and $Y$ are coarsely equivalent provided that either $X$ or $Y$ contains no sparse subspaces consisting of ghostly measured asymptotic expanders. %To the best of our knowledge, this greatly generalises all previously existing results on rigidity for Roe algebras.
Consequently, we obtain the rigidity for all bounded geometry spaces which coarsely embed into some $L^p$-space for $p\in [1,\infty)$. Moreover, we also verify rigidity for the box spaces constructed by Arzhantseva-Tessera and Delabie-Khukhro even though they do \emph{not} coarsely embed into any $L^p$-space.

The key step in our proof of rigidity is showing that a block-rank-one (ghost) projection on a sparse space $X$ belongs to the Roe algebra $C^*(X)$ if and only if $X$ consists of (ghostly) measured asymptotic expanders. As a by-product, we also deduce that ghostly measured asymptotic expanders are new sources of counterexamples to the coarse Baum-Connes conjecture.  
\end{abstract}

\date{\today}
\maketitle

\parskip 4pt

\section{Introduction}\label{sec:introduction}
(Uniform) Roe algebras are $C^*$-algebras associated to metric spaces, which reflect and encode the coarse (or large-scale) geometry of the underlying metric spaces. They have been well-studied and have fruitful applications, among which the most important ones would be the coarse Baum-Connes conjecture, the Novikov conjecture on higher signatures and the
Gromov–Lawson conjecture on positive scalar curvature (see e.g. \cite{WY20} for details).

The main purpose of this paper is to provide a new geometric condition in terms of measured asymptotic expanders to guarantee the rigidity for (uniform) Roe algebras associated to metric spaces. In recent years, the study of the rigidity for those $C^*$-algebras has gained considerable attention (for instance \cite{rigidity_CBC,BragaFarah19,BragaFarahVignati2018,braga2019embeddings,BragaVignati2019,CL18,vspakula2013rigidity,WW18}). Let us recall the general setting: let $(X,d)$ be a metric space with bounded geometry, and  denote by $C_u^*(X)$ the uniform Roe algebra of $X$, defined to be the norm closure of all bounded operators on $\ell^2(X)$ with finite propagation. Similarly the Roe algebra of $X$, denoted by $C^*(X)$, is defined to be the norm closure of all locally compact bounded operators on $\ell^2(X;\H_0)$ with finite propagation, where $\H_0$ is an infinite-dimensional separable Hilbert space. It is well-known that if $X$ and $Y$ are coarsely equivalent metric spaces with bounded geometry, then their (uniform) Roe algebras are (stably) $\ast$-isomorphic. The rigidity problem refers to the converse implication:

\begin{problem}[Rigidity Problem]\label{ProblemRigidity}
Let $X$ and $Y$ be two metric spaces with bounded geometry.
\begin{enumerate}
\item\label{ProblemRigidityUnifRoeAlg}If $C_u^*(X)$ and $C_u^*(Y)$ are stably $*$-isomorphic, are $X$ and $Y$ coarsely equivalent?
\item\label{ProblemRigidityRoeAlg} If $C^*(X)$ and $C^*(Y)$ are $*$-isomorphic, are $X$ and $Y$ coarsely equivalent?
\end{enumerate}
\end{problem}

To briefly summarise the current status of Rigidity Problem (at the time of submission of the present paper\footnote{About half a year after this paper was first announced, an unconditional positive answer
to the rigidity problem for \emph{uniform} Roe algebras is given in \cite{BBFKVW}. However, the method in \cite{BBFKVW} does not immediately apply to Roe algebras. To the best of the authors' knowledge, the rigidity problem for Roe algebras is still open in general.}), and some of the main results of this paper, let us consider the following properties of a metric space with bounded geometry:
\begin{itemize}
  \item[(A)] Property (A).
  \item[(B)] Coarse embeddability into a Hilbert space.
  \item[(C)] Coarse embeddability into some $L^p$-space for $p\in[1,\infty)$.
  \item[(D)] ``All sparse subspaces yield only compact ghost projections in their (uniform) Roe algebras.'' (This technical condition in the uniform case was first introduced in \cite{BragaFarah19,BragaFarahVignati2018, braga2019embeddings}; the general case in \cite{rigidity_CBC}.)
  \item[(E)] The space contains no sparse subspaces consisting of ghostly measured asymptotic expanders. (This geometric condition is introduced in the present paper, see below for precise definitions.)
  \item[(F)] A positive answer to Rigidity Problem \ref{ProblemRigidity}.
\end{itemize}
All of (A) -- (E) imply (F): The implication (A)$\Rightarrow$(F) is the main result in \cite{vspakula2013rigidity} (see Theorem 1.4 and 1.8 therein) where the Rigidity Problem was first posed. The next substantial progress was then made with the help of (D) by Braga and Farah in \cite[Corollary 1.3]{BragaFarah19} (for $*$-isomorphic uniform Roe algebras), but the full generality of (D)$\Rightarrow$(F) was given in \cite[Theorem~1.3]{rigidity_CBC}. One of the main results of this paper is to establish (E)$\Rightarrow$(F) (see Theorem \ref{introthm:rigidity.geometric.condition} below).
% and our result in this paper is the state-of-the-art.}:.
One has the following implications:
$$
\xymatrix{
  & & \text{(C)} \ar@{=>}[r] & \text{(E)} \ar@{=>}[r]  &\text{(F)}\\
  \text{(A)} \ar@{=>}[r] & \text{(B)} \ar@{=>}[r] \ar@{=>}[ur] & \text{(D)} \ar@{=>}[ur]
}
$$

Here (A)$\Rightarrow$(B) is due to Yu \cite[Theorem 2.2]{Yu:CBC-for-HSp:00}. {\color{red}{Also (B)$\Rightarrow$(C) holds trivially{\footnote{We want to recall that condition (B) also implies the coarse embeddability into \emph{any} $L^p$-space for $p\in[1,\infty)$, which was proved in \cite{Nowak:lp-embeddings:06}.}}, however, whether (C) implies (B) is still open for bounded geometry metric spaces.}} Next, (B)$\Rightarrow$(D) is essentially due to Finn-Sell \cite[Proposition 35]{MR3266245}, made precise in \cite[Theorem 5.2]{rigidity_CBC}. The remaining implications (C)$\Rightarrow$(E), (D)$\Rightarrow$(E) and (E)$\Rightarrow$(F) are all established in this paper (see Corollary~\ref{cor: non CE}, Corollary~\ref{cor C} and Theorem \ref{introthm:rigidity.geometric.condition} respectively).

Some of the above implications are known to be \emph{not} reversible: (B)$\not\Rightarrow$(A) by \cite[Theorem 1.1]{AGS:not-A-but-CE:12}, (D)$\not\Rightarrow$(B) by \cite[Proposition~35]{MR3266245} and \cite[Corollary~1.2]{DWY21}. Furthermore, (E)$\not\Rightarrow$(C) by the results in this paper (see Theorem \ref{examples} and Lemma~\ref{lem:measured.weak.embedding} below). Finally, we do not know whether (D)$\Rightarrow$(E) is reversible, and thus whether our new rigidity results here are actually (rather than formally) stronger than those in \cite{rigidity_CBC}. Nevertheless, the condition (E) is geometric and easier to check than the analytic condition (D), which allows us to provide new examples of rigid spaces in Section \ref{sec:examples}, and to establish the condition (C) as another geometric criterion for Rigidity (see Corollary~\ref{introcor:lp}).

To explain our results and approach in more detail, we start with measured asymptotic expanders. They were introduced in \cite[Definition~6.1]{dypartI} and can be naturally constructed from strongly ergodic or asymptotically expanding actions (see \cite[Theorem~6.16]{dypartI}).
\begin{defn}
A sequence of finite probability-measured metric spaces $\{(X_n,d_n,m_n)\}_{n \in \N}$ is
called a sequence of \emph{measured asymptotic expanders} if for any
$\alpha \in (0,\frac{1}{2}]$, there exist $c_\alpha\in (0,1)$ and $R_\alpha>0$
such that for any $n \in \N$ and $A \subseteq X_n$ with
$\alpha \leq m_n(A) \leq \frac{1}{2}$, we have
$m_n(\partial_{R_\alpha} A) > c_\alpha\cdot m_n(A)$, where $\partial_{R_\alpha} A:=\{x\in X_n \backslash A:d_n(x,A)\leq R_\alpha\}$.
\end{defn}

Note that when all probability measures $m_n$ are taken to be normalised counting measures, then we recover the notion of asymptotic expanders introduced in \cite{Intro}.

\begin{defn}\label{DefiMainGeomProp}
Let $(X,d)$ be a metric space. 
\begin{enumerate}
\item $X$ is called \emph{sparse}\footnote{In this case, we also say that $X$ is a \emph{coarse disjoint union} of $X_n$.} if there exists a disjoint partition $X=\bigsqcup_{n\in \N}X_n$ such that
\begin{itemize}
\item $0<|X_n|<\infty$ for all $n\in\N$, and
\item $d(X_n,X_m)\to \infty$ as $n+m\to \infty$ and $n\neq m$.
\end{itemize}
\item We say that a sparse space $X$ \emph{consists of ghostly measured asymptotic expanders} if there exist a decomposition $X=\bigsqcup_nX_n$ as in (1) and a sequence of probability measures $m_n$ on $X_n$ with $\lim_{n\to \infty}\sup_{x\in X_n}m_n(x)=0$ such that $\{(X_n,d_n,m_n)\}_{n \in \N}$ forms a sequence of measured asymptotic expanders, where $d_n$ is the restriction of $d$ to $X_n$.
\end{enumerate}
\end{defn}

We are ready to state our main rigidity theorem:
\begin{introthm}[Theorem \ref{thm:rigidity.geometric.condition}]\label{introthm:rigidity.geometric.condition}
Let $X$ and $Y$ be metric spaces with bounded geometry. Assume that either $X$ or $Y$ contains no sparse subspaces consisting of ghostly measured asymptotic expanders. Then the following are equivalent:
\begin{enumerate}
  \item $X$ is coarsely equivalent to $Y$;
  \item $C^*_u(X)$ is stably $\ast$-isomorphic to $C^*_u(Y)$;
  \item $C^*(X)$ is $\ast$-isomorphic to $C^*(Y)$.
\end{enumerate}
\end{introthm}

Before we discuss examples where the above theorem applies, let us outline our approach to its proof.
The first step is to show that Problem~\ref{ProblemRigidity} has a positive answer under the condition that all sparse subspaces contain no block-rank-one ghost projections in their Roe algebras (see Proposition~\ref{prop:rigidity.Roe})\footnote{We would like to point out that our assumptions in Proposition~\ref{prop:rigidity.Roe} are formally weaker than the technical condition in \cite[Theorem~1.3]{rigidity_CBC}. } in the following sense:

\begin{defn}\label{defn:block rank one projection}
Let $(X,d)=\bigsqcup_{n\in\N}(X_n,d_n)$ be a sparse space and $\H_0$ be a Hilbert space. We say that an operator $P \in \B(\ell^2(X;\H_0))$ is a \emph{block-rank-one projection} with respect to $\{(X_n,d_n)\}_{n\in \N}$ if it is a projection of the form 
$$P:=\bigoplus_{n\in \N} P_n \in \B \left(\bigoplus_{n\in \N}\ell^2(X_n;\H_0)\right)=\B(\ell^2(X;\H_0)),$$
where each $P_n$ is a rank-one projection in $\B(\ell^2(X_n;\H_0))$.

As each $P_n$ has the form $P_n(\eta)=\langle \eta, \xi_n \rangle \xi_n$ for all $\eta \in \ell^2(X_n;\H_0)$ where $\xi_n$ is a unit vector in $\ell^2(X_n;\H_0)$, we define for each $n\in \N$ the \emph{associated probability measure} $m_n$ on $X_n$ by $m_n(x):=||\xi_n(x)||^2$ for $x\in X_n$.

Finally, we say that the block-rank-one projection $P$ is a \emph{ghost}\footnote{See Definition~\ref{defn:ghost.operator} for general ghost operators.} if the sequence of associated probability measures satisfies
$\lim_{n\to \infty}\sup_{x\in X_n}m_n(x)=0$.
\end{defn}

If each $P_n$ is the rank-one projection onto constant functions on $X_n$, then the resulting projection $P$ is the so-called \emph{averaging projection}. Its relation to asymptotic expanders has been extensively studied in \cite{structure,Intro}. The reader is cautioned that the notion of block-rank-one projections depends on the choice of the decomposition of $X$ as a coarse disjoint union of finite metric spaces $X=\bigsqcup_{n\in\N}X_n$, not only on the bijectively coarse equivalence type of $X$ (see \cite[Remark~3.5]{Intro} for details).

The second step is to show that a block-rank-one (ghost) projection $P$ on a sparse space $\bigsqcup_{n\in\N}(X_n,d_n)$ belongs to the Roe algebra $C^*(X)$ if and only if $\{(X_n,d_n,m_n)\}_{n \in \N}$ forms a (ghostly) sequence of measured asymptotic expanders, where each $m_n$ denotes the associated probability measure on $X_n$. To this end, we prove the following theorem, which extends \cite[Theorem~6.1]{structure} from asymptotic expanders to the measured case:

\begin{introthm}[Proposition \ref{prop: measured asymptotic expanders}, Theorem~\ref{thm:main result uniform case} and Theorem \ref{thm:main result Roe}]\label{thm: main result}
Let $(X,d)=\bigsqcup_{n\in\N}(X_n,d_n)$ be a sparse space with bounded geometry and $\H_0$ be a Hilbert space. Let $P \in \B(\ell^2(X;\H_0))$ be a block-rank-one projection with respect to $\{(X_n,d_n)\}_{n\in \N}$ and $m_n$ be the associated probability measure on $X_n$. Then the following hold:
\begin{itemize}
\item If $\H_0=\ell^2(\N)$, then $P$ belongs to the Roe algebra $C^*(X)$ \emph{if and only if} $\{(X_n,d_n,m_n)\}_{n \in \N}$ forms a sequence of measured asymptotic expanders.
\item If $\H_0=\C$, then $P$ belongs to the uniform Roe algebra $C_u^*(X)$ \emph{if and only if} $\{(X_n,d_n,m_n)\}_{n \in \N}$ forms a sequence of measured asymptotic expanders.
\end{itemize}
\end{introthm}

Theorem~\ref{thm: main result} is the fundamental theorem of this paper, and its proof is rather technical. To briefly summarise, we start with a structure theorem inspired by \cite[Theorem 3.5]{structure} that measured asymptotic expanders admit a uniform exhaustion by measured expander graphs with bounded measure ratios (see Corollary~\ref{cor:MAE to ME})\footnote{A dynamical counterpart of structure theorem can be found in \cite{dypartII,dypartI}.}. Next, we show that certain Laplacians associated to measured expanders with bounded measure ratios have spectral gap (see Proposition~\ref{prop: spetral gap for measured expanders}). Unfortunately, the measures involved in measured expanders may not come from reversible random walks on graphs, hence we cannot directly follow the classical argument from the theory of expanders to obtain the required spectral gap. We overcome this issue by building auxiliary random walks, whose stationary measures uniformly control the original measures. Such control is guaranteed when we deal with measured expanders with bounded measure ratios. We refer the reader to Lemma~\ref{lem:measured expanders to reversible ones} for the precise statement.

As outlined above, Theorem~\ref{introthm:rigidity.geometric.condition} can be deduced from our fundamental theorem (Theorem~\ref{thm: main result}). Additionally, Theorem~\ref{thm: main result} has two further substantial consequences. The first one is to provide both analytic and geometric characterisations for our rigidity condition. Namely, we have the following corollary (see also \cite[Problem~8.6]{braga2019embeddings}):

\begin{introcor}[Corollary~\ref{cor:geometric condition characterisation}]\label{cor C}
Let $X$ be a metric space with bounded geometry. Then the following coarse properties are equivalent: 
\begin{enumerate}
\item all sparse subspaces of $X$ contain no block-rank-one ghost projections in their Roe algebras;
\item all sparse subspaces of $X$ contain no block-rank-one ghost projections in their uniform Roe algebras;
\item $X$ contains no sparse subspaces consisting of ghostly measured asymptotic expanders.
\item $X$ coarsely contains no sparse spaces consisting of ghostly measured asymptotic expanders with uniformly bounded geometry.
\end{enumerate}
\end{introcor}

Secondly, Theorem~\ref{thm: main result} possibly gives rise to a new source of counterexamples to the coarse Baum-Connes conjecture. More precisely, we obtain a measured analogue of \cite[Theorem~D]{structure}:
\begin{introcor}[Corollary~\ref{cbc+gmae}]
Let $X$ be a sparse space consisting of ghostly measured asymptotic expanders. If $X$ has bounded geometry and admits a fibred coarse embedding into a Hilbert space, then it does not satisfy the coarse Baum-Connes conjecture. 
\end{introcor}

The remainder of the paper concentrates on the study of measured (asymptotic) expanders in order to provide sufficient conditions and examples for Corollary~\ref{cor C} (and thus also Theorem~\ref{introthm:rigidity.geometric.condition}). In Section~\ref{subsection: avgeraging projection of measured expanders}, we obtain an $L^p$-Poincar\'{e} inequality for measured expanders (see Corollary~\ref{cor:Poincare for measured expanders}). Using the structure theorem for measured asymptotic expanders (Corollary~\ref{cor:MAE to ME}), we conclude that any sparse space consisting of ghostly measured asymptotic expanders with uniformly bounded geometry cannot coarsely embed into any $L^p$-space for $p\in [1,\infty)$ (see Corollary~\ref{cor: non CE}). Thus, we obtain the following corollary of Theorem~\ref{introthm:rigidity.geometric.condition}:

\begin{introcor}\label{introcor:lp}
Let $X$ and $Y$ be metric spaces with bounded geometry such that $Y$ coarsely embeds into some $L^p$-space for $p \in [1,\infty)$. Then the following are equivalent:
\begin{enumerate}
  \item $X$ is coarsely equivalent to $Y$;
  \item $C^*_u(X)$ is stably $\ast$-isomorphic to $C^*_u(Y)$;
  \item $C^*(X)$ is $\ast$-isomorphic to $C^*(Y)$.
\end{enumerate}
\end{introcor}

Apart from Corollary~\ref{introcor:lp}, we provide concrete examples of spaces which \emph{are} rigid, but \emph{do not} coarsely embed into any $L^p$-space. These examples are \emph{not} covered by previously existing results at least for Roe algebras.

To this end, we come up with the concept of a \emph{measured weak embedding} (see Definition~\ref{de:measure.weak.embedding}) and extend the main results in \cite{AT15, delabie2018box} to the measured setting. More precisely, the structure result (Corollary~\ref{cor:MAE to ME}) together with the $L^p$-Poincar\'{e} inequality for measured expanders (Corollary~\ref{cor:Poincare for measured expanders}) leads to Proposition~\ref{prop:AT15 analogue}, which is a measured analogue of \cite[Proposition 2]{AT15}. By virtue of Proposition~\ref{prop:AT15 analogue}, the same idea in \cite{AT15, delabie2018box} can be easily adapted to show the following theorem:

\begin{introthm}[Theorem~\ref{AT example} and Theorem~\ref{box space of F3}]\label{examples}
There exists a box space $X$ of a finitely generated residually finite group $\Gamma$ such that $X$ does not coarsely embed into any $L^p$-space for $1\leq p<\infty$, but $X$ does not measured weakly contain any measured asymptotic expanders with uniformly bounded geometry. In fact, the group $\Gamma$ may also be chosen to be the free group $F_3$ of rank three. 
\end{introthm}

As the box space in Theorem~\ref{examples} obviously cannot contain sparse subspaces consisting of ghostly measured asymptotic expanders (see Lemma~\ref{lem:measured.weak.embedding}), we can apply Theorem~\ref{introthm:rigidity.geometric.condition} to produce new rigid spaces:

\begin{introcor}\label{introcor:not-in-Lp-but-rigid}
	There exist box spaces with bounded geometry that do not coarsely embed into any $L^p$-space for $1\leq p<\infty$, but affirmatively answer Problem~\ref{ProblemRigidity}.
\end{introcor}

In \cite[Thereom~1.2 and Thereom~1.3]{AT18}, Arzhantseva and Tessera constructed two examples of finitely generated groups as split extensions of two groups coarsely embeddable into Hilbert space but which themselves do not coarsely embed into Hilbert space. Hence Problem~\ref{ProblemRigidity} also has a positive answer for these two finitely generated groups, as they contain no sparse subspaces consisting of ghostly measured asymptotic expanders by Proposition~\ref{prop:AT15 analogue}. However, the rigidity for these groups is \emph{not} new at all and has already been verified in \cite{rigidity_CBC}. Indeed, we simply apply \cite[Theorem~1.5]{rigidity_CBC} because these two groups satisfy the coarse Baum-Connes conjecture \emph{with coefficients} (see \cite[Remark~2.10 and Proposition~2.11]{rigidity_CBC}). %On the other hand, we do not know whether the box spaces in Theorem~\ref{examples} satisfy the coarse Baum-Connes conjecture %(see \cite[Section~8]{AT15}).

We end the introduction by asking the following two questions:
\begin{question}
Are ghostly measured asymptotic expanders obstacles to Problem~\ref{ProblemRigidity} (2)?
\end{question}

\begin{question}
In the last decade, a theory of ``high dimensional expanders'' has begun to emerge (see \cite{MR3966743} for a detailed survey). Therefore, it is natural to ask: can we define \emph{``measured asymptotic expanders in high dimensions''} such that the corresponding Theorem~\ref{thm: main result} holds for all block-rank-$n$ projections for $n\in \N$?
\end{question}

\subsection*{Convention} Throughout the paper, all metric spaces are \emph{non-empty and discrete}, and graphs are \emph{connected and undirected}.

%\subsection*{Declarations} Funding: KL was supported by the Internal KU Leuven BOF project C14/19/088. KL has also received funding from the European Research Council (ERC) under the European Union's Horizon 2020 research and innovation programme (grant agreement no. 677120-INDEX). J\v{S} was partially supported by Marie Curie FP7-PEOPLE-2013-CIG Coarse Analysis (631945). JZ was supported by NSFC11871342. Conflicts of interest: On behalf of all authors, the corresponding author states that there is no conflict of interest. Availability of data and material: Not
%applicable. Code availability: Not applicable. Authors’ contributions: Not applicable.

\section{Preliminaries}\label{Preliminaries}
In this section, we introduce notation for metric spaces, and recall several basic definitions and results from coarse geometry.

\subsection{Metric spaces and expander graphs}\label{ssec:metric spaces}
Let $(X,d)$ be a metric space, $x\in X$ and $R>0$. Then $B(x,R)$ denotes the
\emph{closed ball} with radius $R$ and centre at $x$. For any subset $A \subseteq X$, $|A|$ denotes the \emph{cardinality} of $A$; $\diam(A)$ denotes the \emph{diameter} of $A$; $\Nd_R(A)=\{x\in X : d(x,A)\leq R\}$ denotes the \emph{$R$-neighbourhood of $A$}; and $\partial_R A=\{x\in X\backslash A: d(x,A)\leq R\}$ denotes the \emph{(outer) $R$-boundary of $A$}. For a subspace $Y \subseteq X$ and a subset $A \subseteq Y$, we may also consider the  \emph{relative} neighbourhood and boundary defined as $\Nd_R^Y(A)=\Nd_R(A) \cap Y$ and $\partial_R^Y A = (\partial_R A) \cap Y$.

Moreover, we say that a metric space $(X,d)$ has \emph{bounded geometry} if $\sup_{x\in X} |B(x,R)|$ is finite for each $R>0$. Any metric space with bounded geometry is automatically countable and discrete. More generally, a sequence of metric spaces $\{(X_n,d_n)\}_{n\in \N}$ has \emph{uniformly bounded geometry} if for all $R>0$, the number $N_R:=\sup_{n\in \N}\sup_{x\in X_n}|B(x,R)|$ is finite. 

In this paper, we are particularly interested in a certain type of metric spaces arising from graphs, the so-called expander graphs.

For a (connected) graph $\G=(V,E)$ with the vertex set $V$ and the edge set $E$, we endow $V$ with the \emph{edge-path} metric, which is defined to be the length (\emph{i.e.}, the number of edges) in a shortest path connecting given two vertices. For $A\subseteq V$, $\partial^V A$ (or just $\partial A$) denotes its $1$-boundary, which is also called the \emph{vertex boundary} of $A$, and its \emph{edge boundary} $\partial^E A$ is defined to be the set of all edges in $E$ with one endpoint in $A$ and the other one in $V \setminus A$.

We say that two vertices $v$ and $w$ in $V$ are \emph{adjacent} if there is an edge in $E$ connecting them (in symbols $v \thicksim_E w$ or just $v \thicksim w$). For a vertex $v\in V$, its \emph{valency} is defined to be the number of vertices adjacent to $v$. It is clear that the edge-path metric on $V$ has bounded geometry if and only if the graph has \emph{bounded valency} in the sense that there exists some $K \geq 0$ such that each vertex has valency bounded by $K$. More generally, a sequence of graphs $\{\G_n\}_{n\in \N}$ is said to have \emph{uniformly bounded valency} if there exists some $K \geq 0$ such that each vertex in $\G_n$ has valency bounded by $K$ for all $n$. Hence, a sequence of graphs has uniformly bounded valency if and only if it has uniformly bounded geometry with respect to the edge-path metrics. 

For a finite graph $\G=(V,E)$, its \emph{Cheeger constant} is defined to be
\[
c(\G):=\min\big\{\frac{|\partial A|}{|A|}: A\subseteq V,\  0<|A|\leq \frac{|V|}{2}\big\},
\]
where $\partial A$ is the vertex boundary of $A$. Now we recall the definition of expander graphs, which are highly connected but sparse at the same time.
\begin{defn}\label{defn:expanders}
A sequence of finite (connected)\footnote{We note that the expanding condition already implies that all the graphs $\G_n$ in an expander sequence are connected.} graphs $\{\G_n=(V_n,E_n)\}_{n\in \N}$ is called a sequence of \emph{expander graphs} if it has uniformly bounded valency, $|V_n| \to \infty$ as $n\to \infty$, and $\inf_{n\in \N} c(\G_n)>0$.
\end{defn}

We can always make a sequence of finite metric spaces into a single metric space, which is sparse in the sense of Definition~\ref{DefiMainGeomProp}:
 
 \begin{defn}\label{defn:coarse disjoint union}
Let $\{(X_n,d_n)\}_{n\in \N}$ be a sequence of finite metric spaces. A \emph{coarse disjoint union} of $\{(X_n,d_n)\}_{n\in \N}$ is a metric space $(X,d)$, where $X=\bigsqcup_{n\in\N}X_{n}$ is the disjoint union of $\{X_n\}$ as a set and $d$ is a metric on $X$ such that
\begin{itemize}
	\item the restriction of $d$ on each $X_n$ coincides with $d_n$;
	\item $d(X_n,X_m)\rightarrow \infty $ as $n+m\rightarrow \infty$ and $n\neq m$.
\end{itemize}
If we need a precise choice of $d$, we can choose $d$ as follows: $d(x,y)=n+m+\diam(X_n)+\diam(X_m)$ for all $x\in X_n$ and $y\in X_m$, whenever $m$ and $n$ are distinct natural numbers. Moreover, any two of such metrics are coarsely equivalent (see Section~\ref{Metric embeddings} for the precise meaning).
\end{defn}
Clearly, a sequence of finite metric spaces $\{X_n\}_{n\in \N}$ has uniformly bounded geometry \emph{if and only if} its coarse disjoint union has bounded geometry.

\subsection{Metric embeddings}\label{Metric embeddings}

Given two metric spaces $(X,d_X)$ and $(Y,d_Y)$, a map $f: X \to Y$ is called a \emph{coarse embedding} if there exist non-decreasing unbounded functions $\rho_\pm: [0,\infty) \to [0,\infty)$ such that 
\[
\rho_-(d_X(x,y)) \leq d_Y(f(x),f(y)) \leq \rho_+(d_X(x,y))
\]
for any $x,y\in X$. In this case, we simply say that $Y$ \emph{coarsely contains} $X$. If $f:X\to Y$ is a coarse embedding and $\Nd_C(f(X))=Y$ for some constant $C > 0$, then we say that $f$ is a \emph{coarse equivalence} between $X$ and $Y$. In this case, we also say that $X$ and $Y$ are \emph{coarsely equivalent}.

More generally, let $\{X_n\}_{n\in \N}$ be a sequence of metric spaces and let $Y$ be another metric space. A sequence of maps $f_n:X_n \to Y$ is called a \emph{coarse embedding of $\{X_n\}_{n\in \N}$ into $Y$} if all $f_n$ are coarse embeddings with the same control functions $\rho_{\pm}$. One can easily check that a sequence of \emph{finite} metric spaces admits a coarse embedding into a non-trivial Banach space $Y$ if and only if its coarse disjoint union as a single metric space admits a coarse embedding into $Y$. \newline

Following the idea in \cite{MR1978492} (see also \cite[Section~7]{MR1911663} and \cite[Definition~5.2]{MR2058475}), we say that a sequence of finite graphs $\{X_n\}_{n\in \N}$ admits a \emph{weak embedding} into a metric space $Y$ if there exist $L>0$ and $L$-Lipschitz maps $f_n\colon X_n\to Y$ such that for every $R>0$,
\begin{align}\label{weak embed}
\lim_{n\to\infty}\sup_{x\in X_n}\frac{|f_n^{-1}(B(f_n(x),R))|}{|X_n|}=0.
\end{align}
In this case, we also say that $Y$ \emph{weakly contains} the sequence $\{X_n\}_{n\in \N}$. 

\begin{rem}\label{bounded geometry weak embed}
If the target space $Y$ has bounded geometry, then (\ref{weak embed}) is equivalent to the following (see \cite[Definition~5.2]{MR2058475}):
$$
\lim_{n\to\infty}\sup_{x\in X_n}\frac{|f_n^{-1}(f_n(x))|}{|X_n|}=0.
$$
Thus, a coarse embedding of a sequence of finite graphs $\{X_n\}_{n\in \N}$ with uniformly bounded valency into a metric space $Y$ is a weak embedding provided that $|X_n|\to \infty$ as $n\to \infty$.
\end{rem}

For the purpose of this paper, we would also like to consider measured weak embeddings of ghostly sequences of finite measured metric spaces.

\begin{defn}\label{def:ghostly}
We say that $\{(X_n,d_n,m_n)\}_{n\in \N}$ is a sequence of \emph{finite measured metric spaces} if each $(X_n,d_n)$ is a finite metric space, and each $m_n$ is a non-trivial and finite measure defined on the $\sigma$-algebra of all subsets of $X_n$. Moreover, the sequence $\{(X_n,d_n,m_n)\}_{n\in \N}$ is called \emph{ghostly} if
\begin{align}\label{ghostly}
\lim_{n\to \infty} \sup_{x\in X_n} \frac{m_n(x)}{m_n(X_n)}=0.
\end{align}
\end{defn}

\begin{rem}\label{rem:ghostly}
If a sequence of finite measured metric spaces $\{(X_n,d_n,m_n)\}_{n\in \N}$ is ghostly, then $|\supp (m_n)|\to \infty$ as $n\to \infty$. However, it is clear that the converse may fail in general. Moreover, if each $m_n$ is the counting measure on $X_n$ then $\{(X_n,d_n,m_n)\}_{n\in \N}$ is ghostly if and only if $|X_n|\to \infty$ as $n\to \infty$.
\end{rem}

\begin{defn}\label{de:measure.weak.embedding} 
We say that a sequence of finite measured metric spaces $\{(X_n,d_n,m_n)\}_n$ admits a \emph{measured weak embedding} into a metric space $Y$ if there exists a sequence of maps $f_n\colon X_n\to Y$ such that
\begin{itemize}
\item[(1)] the sequence $\{f_n\}_{n\in \N}$ is \emph{coarse}\footnote{We follow the terminology from \cite{AT15}; other sources use \emph{bornologous} or \emph{uniformly expansive} for ``coarse'', and further \emph{uniformly coarse} or \emph{equi--coarse} when used for families for maps.}, \emph{i.e.}, there exists a non-decreasing unbounded function $\rho_+: [0,\infty) \to [0,\infty)$ such that for every $n\in \N$ and $x,y\in X_n$, we have $d_Y(f_n(x),f_n(y))\leq \rho_+(d_{X_n}(x,y))$;
\item[(2)] for every $R>0$,
\begin{equation}\label{EQ:measured weak embedding}
  \lim_{n\to\infty}\sup_{x\in X_n}\frac{m_n(f_n^{-1}(B(f_n(x),R)))}{m_n(X_n)}=0.
\end{equation}
\end{itemize}
In this case, we also say that $Y$ \emph{measured weakly contains} the sequence $\{(X_n,d_n,m_n)\}_n$.
\end{defn}

\begin{rem}\label{graph mwe}
When all $X_n$ in Definition~\ref{de:measure.weak.embedding} are finite graphs, condition (1) is equivalent to that all maps $f_n:X_n\to Y$ are $L$-Lipschitz for some $L>0$. Moreover, if each $m_n$ is the counting measure on $X_n$, then we recover the usual notion of weak embedding.
\end{rem}

Similar to Remark~\ref{bounded geometry weak embed}, if the target space $Y$ has bounded geometry then (\ref{EQ:measured weak embedding}) is equivalent to the following:
$$
\lim_{n\to\infty}\sup_{x\in X_n}\frac{m_n(f_n^{-1}(f_n(x)))}{m_n(X_n)}=0.
$$
At this point, the reader might want to compare this expression with (\ref{ghostly}).

Recall that a map $f:X\to Y$ between two metric spaces is called \emph{finite-to-one} if the preimage $f^{-1}(y)$ is finite for every $y\in Y$. Moreover, a sequence of maps $\{f_n: X_n\to Y\}_{n\in \N}$ between metric spaces is called \emph{uniformly finite-to-one} if there exists some $N\in \N$ such that for any $n\in \N$ and any $y\in Y$ we have $|f_n^{-1}(y)| \leq N$. In particular, we say that a map $f:X\to Y$ between two metric spaces is \emph{uniformly finite-to-one} if the constant sequence $\{f_n=f\}_{n\in \N}$ is uniformly finite-to-one. Notice that if $X$ has bounded geometry, then any coarse embedding $f: X\to Y$ is uniformly finite-to-one.

The proof of the following lemma is straightforward from the definitions, thus omitted.
\begin{lem}\label{lem:measured.weak.embedding}
Let $\{(X_n,d_n,m_n)\}_{n\in \N}$ be a ghostly sequence of finite measured metric spaces, and $Y$ be a metric space. Then the following hold:
\begin{itemize}
\item If $\{(X_n,d_n)\}_n$ has uniformly bounded geometry, then every coarse embedding of the sequence $\{(X_n,d_n,m_n)\}_n$ into $Y$ is a measured weak embedding.
\item If $Y$ has bounded geometry, then every uniformly finite-to-one and coarse sequence of maps from $\{(X_n,d_n,m_n)\}_n$ into $Y$ is a measured weak embedding. 
\end{itemize}
\end{lem}

\subsection{Rigidity for Roe-like algebras}
For a Hilbert space $\H$, we denote the closed unit ball of $\H$ by $(\H)_1$. Moreover, $\B(\H)$ and $\K(\H)$ denote the spaces of bounded and compact operators on $\H$, respectively.

For a discrete metric space $(X,d)$, we denote by $\chi_A$ the characteristic function on a subset $A\subseteq X$ and $\delta_x=\chi_{\{x\}}$ for $x\in X$. Fixing a Hilbert space $\H_0$, we consider the Hilbert space $\ell^2(X;\H_0) \cong \ell^2(X) \otimes \H_0$. When $\H_0=\C$, we also write $\ell^2(X)=\ell^2(X;\C)$. There is a natural $*$-representation $\ell^\infty(X) \to \B(\ell^2(X;\H_0))$ induced by the canonical multiplication representation of $\ell^\infty(X)$ on $\ell^2(X)$. 

We can always regard any operator $T \in \B(\ell^2(X;\H_0))$ as an $X$-by-$X$ matrix $(T_{x,y})_{x,y\in X}$ with entries in $\B(\H_0)$. More precisely, $T_{x,y}\in \B(\H_0)$ is defined by 
\[
T_{x,y}\xi = \big(T(\delta_y \otimes \xi)\big) (x)
\]
for $\xi\in \H_0$ and $x,y\in X$. (Here $\delta_y\otimes\xi\in \ell^2(X) \otimes \H_0 \cong \ell^2(X;\H_0)$ represents the function from $X$ to $\H_0$ taking the value $\xi$ at  $y$ and $0$ otherwise.) The \emph{propagation of $T$} is defined to be
\[
\ppg(T):=\sup\{d(x,y): T_{x,y} \neq 0\}.
\]
We say that $T\in \B(\ell^2(X;\H_0))$ has \emph{finite propagation} if $\ppg(T)$ is finite, and $T$ is \emph{locally compact} if $T_{x,y} \in \K(\H_0)$ for all $x,y\in X$.

In this paper, we will deal with the rigidity problem for the following variants of Roe algebras:
\begin{defn}\label{defn:Roe algebra}
Let $X$ be a metric space with bounded geometry and $\H_0$ be an infinite-dimensional separable Hilbert space.
\begin{enumerate}
  \item The \emph{Roe algebra of $X$}, denoted by $C^*(X)$, is defined to be the norm closure of all finite propagation locally compact operators in $\B(\ell^2(X;\H_0))$.
  \item The \emph{uniform algebra of $X$}, denoted by $UC^*(X)$, is defined to be the norm closure of all finite propagation operators $T\in \B(\ell^2(X;\H_0))$ such that there exists some $N \in \N$ satisfying $\mathrm{rank}(T_{x,y}) \leq N$ for all $x,y\in X$.
  \item The \emph{stable Roe algebra of $X$}, denoted by $C^*_s(X)$, is defined to be the norm closure of all finite propagation operators $T\in \B(\ell^2(X;\H_0))$ such that there exists a finite-dimensional subspace $H \subseteq \H_0$ satisfying $T_{x,y} \in \B(H)$ for all $x,y\in X$.
  \item The \emph{uniform Roe algebra of $X$}, denoted by $C^*_u(X)$, is defined to be the norm closure of all finite propagation operators $T \in \B(\ell^2(X))$.
\end{enumerate}
\end{defn}

It is known that $C^*_u(X) \hookrightarrow C^*_s(X) \subseteq UC^*(X) \subseteq C^*(X)$, where the last two inclusions are canonical. Moreover, we have $C^*_u(X) \otimes \K(\H_0) \cong C^*_s(X)$. We also use the following simplified terminology:

\begin{defn}[{\cite[Definition 2.3]{rigidity_CBC}, but see also \cite{BragaFarah19}}]\label{defn:Roe like algebra}
Let $X$ be a metric space with bounded geometry. A $C^*$-subalgebra $A \subseteq C^*(X)$ is called \emph{Roe-like} if $C^*_s(X) \subseteq A$.
\end{defn}

The rigidity for Roe-like algebras has already been extensively studied in the literature (\emph{e.g.}, \cite{BragaFarah19, braga2019embeddings, rigidity_CBC, vspakula2013rigidity}). In order to formulate the most up-to-date rigidity result, we need to recall several concepts: 

\begin{defn}[G. Yu\footnote{Attributed to G.~Yu in \cite[11.5.2]{Roe:lectures-on-coarse-geometry:03}.}]\label{defn:ghost.operator}
Let $X$ be a discrete metric space, $\H_0$ be a Hilbert space and $T \in \B(\ell^2(X;\H_0))$. We say that $T$ is a \emph{ghost} if for any $\varepsilon>0$, there exists a bounded subset $F \subseteq X$ such that for any $(x,y) \notin F \times F$ we have $\|T_{x,y}\|<\varepsilon$.
\end{defn}

The following lemma clarifies Definition~\ref{defn:block rank one projection} and Definition~\ref{def:ghostly}. Since its proof is elementary, we leave the details to the reader:
\begin{lem}\label{lem: ghost}
Let $(X,d)=\bigsqcup_{n\in\N}(X_n,d_n)$ be a sparse space and $\H_0$ be a Hilbert space. If $P \in \B(\ell^2(X;\H_0))$ is a block-rank-one projection with respect to $\{(X_n,d_n)\}_{n\in \N}$, then $P$ is a ghost \emph{if and only if} the associated sequence of finite measured metric spaces $\{(X_n,d_n,m_n)\}_n$ is ghostly.
\end{lem}

\begin{defn}\label{technical consdition}
Let $X$ be a metric space with bounded geometry. We say that \emph{all sparse subspaces of $X$ yield only compact ghost projections in their Roe algebras} if for any sparse subspace $X'\subseteq X$, all ghost projections in $C^*(X') $ are compact.
\end{defn}

\begin{thm}[{\cite[Theorem 1.3]{rigidity_CBC}}]\footnote{We also refer the reader to \cite[Corollary 1.3]{BragaFarah19}.}\label{thm:rigidity_CBC}
Let $X$ and $Y$ be metric spaces with bounded geometry. If all sparse subspaces of $Y$ yield only compact ghost projections in their Roe algebras, then the following are equivalent:
\begin{enumerate}
  \item $X$ is coarsely equivalent to $Y$;
  \item $C^*_u(X)$ is stably $\ast$-isomorphic to $C^*_u(Y)$;
  \item $C^*_s(X)$ is $\ast$-isomorphic to $C^*_s(Y)$;
  \item $UC^*(X)$ is $\ast$-isomorphic to $UC^*(Y)$;
  \item $C^*(X)$ is $\ast$-isomorphic to $C^*(Y)$.
\end{enumerate}
\end{thm}
As shown in \cite[Theorem~5.3]{rigidity_CBC},  all sparse subspaces of a metric space $Y$ yield only compact ghost projections in their Roe algebras if $Y$ satisfies various forms of the Baum-Connes conjecture. In particular, rigidity holds for those spaces.

\section{Rigidity}
In this section, we introduce new analytic conditions which guarantee rigidity for Roe-like algebras. More precisely, we define the following:
\begin{defn}\label{our geometric condition}
Let $X$ be a metric space with bounded geometry. We say that \emph{all sparse subspaces of $X$ contain no block-rank-one ghost projections in their Roe algebras (or in their uniform Roe algebras)} if $C^*(X')$ (or $C_u^*(X')$) contains no block-rank-one ghost projections for every sparse subspace $X'\subseteq X$.
 \end{defn}

In the remainder of this section, we will follow an approach as in \cite{rigidity_CBC} to show that the above new analytic condition concerning Roe algebras in Definition~\ref{our geometric condition} is already sufficient for rigidity. First of all, we need the following lemma which is a combination of \cite[Theorem~7.4]{BragaFarah19} and \cite[Lemma~3.1]{rigidity_CBC}.

\begin{lem}\label{LemmaRankProjInBDelta}
Let $(Y,d)$ be a metric space with bounded geometry, and assume that all sparse subspaces of $Y$ contain no block-rank-one ghost projections in their Roe algebras. Let $(p_n)_n$ be an orthogonal sequence of rank-one projections such that $\sum_{n\in M }p_n$ converges in the strong operator topology to an element in $C^*(Y)$ for all $M\subseteq \N$. Then
\[ \inf_{n\in \N}\sup\big\{ \|p_n\delta_y\otimes v\|: y\in Y,\ v\in (\H_0)_1 \big\}>0.\]
\end{lem}

\begin{proof}
The proof is almost the same as the one for \cite[Lemma~3.1]{rigidity_CBC}, except that we need the following claim instead of Claim 3.3 therein:
\begin{claim*}\label{Claim1}
By going to a subsequence of $(p_n)_n$, there exists a sequence $(Y_n)_n$ of disjoint finite non-empty subsets of $Y$ and a sequence of rank-one projections $(q_n)_n$ in $\cstr(Y)$ such that 
\begin{enumerate}
\item $d(Y_k,Y_m)\to \infty$ as $k+m\to \infty$ and $k\neq m$; 
\item $\|p_n-q_n\|<2^{-n}$.
\end{enumerate}
\end{claim*} 
The proof of this claim is very similar to Claim 3.3 in \cite[Lemma~3.1]{rigidity_CBC}, except that we require each $q_n$ to be rank-one rather than finite rank. But this follows automatically from the functional calculus construction $q_n=f(a)$ therein, where $a$ is a rank-one self-adjoint operator and $f$ is a continuous function on the spectrum of $a$ with $f(0)=0$. In particular, the rank of $q_n$ is bounded by one and exactly equals to one by (2). For more details we refer the reader to \cite[Lemma~3.1]{rigidity_CBC}.
\end{proof}

As a direct consequence of Lemma~\ref{LemmaRankProjInBDelta}, we obtain the following corollary (\emph{cf.} \cite[Corollary~3.5]{rigidity_CBC}):
\begin{cor}\label{cor:3.5 in BCL}
Let $X$ and $Y$ be metric spaces with bounded geometry, and assume that all sparse subspaces of $Y$ contain no block-rank-one ghost projections in their Roe algebras. If $A\subseteq C^*(X)$ is a Roe-like $C^*$-subalgebra, then every strongly continuous rank-preserving $*$-homomorphism $\Phi:A \rightarrow C^*(Y)$ is a \emph{rigid $*$-homomorphism} in the following sense:
\begin{align*}
\sup_{u\in (\H_0)_1}\inf_{x\in X}\sup_{y\in Y,\  v\in (\H_0)_1}\|\Phi(e_{(x,u),(x,u)})\delta_y\otimes v\|>0.
\end{align*}
\end{cor}

The next proposition slightly extends \cite[Theorem 4.5]{rigidity_CBC}.

\begin{prop}\label{prop:4.5 in BCL}
Let $X$ and $Y$ be metric spaces with bounded geometry, and assume that all sparse subspaces of $Y$ contain no block-rank-one ghost projections in their Roe algebras. Let $A\subseteq C^*(X)$ and $B\subseteq C^*(Y)$ be Roe-like $C^*$-algebras such that either $B=C^*_s(Y)$ or $\mathrm{UC}^*(Y)\subseteq B$. If $A$ embeds onto a hereditary $C^*$-subalgebra of $B$, then $X$ coarsely embeds into $Y$.
\end{prop}

\begin{proof}
Let $\Phi:A\to B$ be an embedding onto a hereditary $C^*$-subalgebra of $B$. By \cite[Lemma~4.1]{rigidity_CBC}, $\Phi$ is strongly continuous and rank-preserving.  By Corollary~\ref{cor:3.5 in BCL}, $\Phi$ is a rigid $*$-homomorphism. It follows from \cite[Lemma~4.2 and Lemma~4.4]{rigidity_CBC} that $\Phi$ induces a coarse embedding from $X$ into $Y$ (see also the proof of \cite[Theorem~4.5]{rigidity_CBC} for more details).
\end{proof}

We also need the following key proposition:
\begin{prop}\label{prop:general case of projections in Roe}
Let $f: (X,d_X) \to (Y,d_Y)$ be a finite-to-one coarse map between metric spaces with bounded geometry and $\H_0$ be an infinite-dimensional separable Hilbert space. Assume that there exist a sparse subspace $X_0=\bigsqcup_{n\in \N} X_n$ of $X$ and a block-rank-one projection $P=\bigoplus_{n\in \N} P_n \in \B(\ell^2(X_0;\H_0))$ with respect to $\{X_n\}_{n\in \N}$ such that $P\in C^*(X_0)$.

Let $m_n$ be the associated probability measure on $X_n$ given by $m_n(x)=||P_n\delta_x||^2$ for $x\in X_n$. Denote $f_n:=f|_{X_n}$ and $d_n:=d_X|_{X_n}$. If the sequence $\{f_n:(X_n,d_n,m_n) \to Y\}_{n\in \N}$ is a measured weak embedding, then there exist a sparse subspace $Y' \subseteq Y$ and a block-rank-one ghost projection in $C^*(Y')$.
\end{prop}

\begin{rem}
  We do not know whether the conclusion holds also in the case when $\H_0$ is finite-dimensional.
\end{rem}
  
\begin{proof}
Since the map $f$ is finite-to-one and $Y$ has bounded geometry, it is easy to show that there is a subsequence $\{X_n\}_{n\in M}$ such that $d_Y(f(X_n),f(X_m)) \to \infty$ as $n+m \to \infty$ for $n \neq m$, where $M$ is an infinite subset of $\N$. If we take $Y_n:=f(X_n)$ for each $n\in M$ and $Y':=\bigsqcup_{n\in M}Y_n$ to be their disjoint union, then $Y'$ is a sparse subspace of $Y$. We consider the sparse subspace $X':=\bigsqcup_{n\in M} X_n$ and the block-rank-one projection $P':=\bigoplus_{n\in M} P_n \in \B(\ell^2(X';\H_0))$. Since $P'=\chi_{X'}P\chi_{X'}$ and $\chi_{X'}$ has finite propagation, we deduce that $P'\in C^*(X')$.

Let $\H_{X'}:=\ell^2(X';\H_0) \cong \ell^2(X') \otimes \H_0$ and $\H_{Y'}:=\ell^2(Y'; \H_{X'}) \cong \ell^2(Y') \otimes \ell^2(X') \otimes \H_0$. Similarly to the proof of \cite[Proposition 6.4]{structure}, we construct an isometry $V: \H_{X'} \to \H_{Y'}$ covering the coarse map $f':=f|_{X'}: X' \to Y'$ in the sense that $\supp (V) \subseteq \{(f'(x),x)\in Y'\times X' ~|~ x \in X'\}$, where $\supp (V):=\{(y,x)\in Y'\times X' ~|~ V_{y,x}\neq 0\}$.

Indeed, for each $y \in Y'$ we define an isometry $V_y: \ell^2(f'^{-1}(y)) \otimes \H_0 \to \C \delta_{y}\otimes \ell^2(X') \otimes \H_0$ by the formula
\[
\delta_z \otimes \xi \mapsto \delta_{f'(z)} \otimes \delta_z \otimes \xi = \delta_y \otimes \delta_z \otimes \xi, \ \text{for $z\in f'^{-1}(y)$ and $\xi\in \H_0$}.
\]
Since $\H_{X'}=\bigoplus_{y\in Y'} \left(\ell^2(f'^{-1}(y)) \otimes \H_0 \right)$ and $\H_{Y'}=\bigoplus_{y\in Y'}\left( \C \delta_{y}\otimes \ell^2(X') \otimes \H_0 \right)$, we can  define $V:= \bigoplus_{y\in Y'} V_y: \H_{X'}\longrightarrow \H_{Y'}$. By the construction of $V$, we clearly have that $\supp (V) \subseteq \{(f'(x),x)\in Y'\times X' ~|~ x \in X'\}$.

As $V$ is a covering isometry for the coarse map $f'$, it follows that the isometry $V$ induces a $*$-homomorphism $\mathrm{Ad}_V: C^*(X') \to C^*(Y')$ given by $T \mapsto VTV^*$ (see also \cite[Lemma~5.1.12 and Remark~5.1.13]{WY20}). Hence, $Q:= VP'V^*$ is  a projection in the Roe algebra $C^*(Y')$. As $f'=\bigsqcup_{n\in M} f_n$, we have that $V=\bigoplus_{n\in M}V_n$ where
\begin{align*}
V_n:=\bigoplus_{y\in Y_n}V_y: \ell^2(X_n; \H_0)\to \ell^2(Y_n; \H_{X'}).
\end{align*}
Thus $Q_n:=V_nP_nV_n^*$ is a rank-one projection in $\B(\ell^2(Y_n;\H_{X'}))$ and $Q=\bigoplus_{n\in M}Q_n$ is block-rank-one with respect to $\{Y_n\}_{n\in M}$.

Finally, we verify that $Q$ is a ghost. A direct calculation shows that for any $y,z \in Y_n$, we have that $(Q_n)_{y,z}=\chi_{f_n^{-1}(y)}P_n\chi_{f_n^{-1}(z)}$. Moreover,
\begin{align}\label{formula of measure on Yn}
\|(Q_n)_{y,z}\|=\sqrt{m_n(f_n^{-1}(y))\cdot m_n(f_n^{-1}(z))}.
\end{align}
Thus, it follows easily that $Q$ is a ghost as $\{f_n:(X_n,d_n,m_n) \to Y\}_{n\in M}$ forms a measured weak embedding.  
\end{proof}

Combining Proposition~\ref{prop:general case of projections in Roe} with Lemma~\ref{lem:measured.weak.embedding} and Lemma~\ref{lem: ghost}, we obtain the following (\emph{cf.} \cite[Theorem~7.6 and Remark~7.8]{braga2019embeddings}):

\begin{cor}\label{prop:coarse invariant}
Let $X$ and $Y$ be metric spaces with bounded geometry, and let $f: X \to Y$ be a uniformly finite-to-one coarse map. If all sparse subspaces of $Y$ contain no block-rank-one ghost projections in their Roe algebras, then the same holds for $X$.
\end{cor}

As every coarse embedding from a metric space with bounded geometry into another metric space is uniformly finite-to-one, Proposition~\ref{prop:4.5 in BCL} and Corollary~\ref{prop:coarse invariant} together give rise to the following corollary (\emph{cf.} \cite[Corollary~4.6]{rigidity_CBC}):

\begin{cor}\label{Cor4.6 in BCL}
Let $X$ and $Y$ be metric spaces with bounded geometry, and assume that all sparse subspaces of $Y$ contain no block-rank-one ghost projections in their Roe algebras. Let $A\subseteq C^*(X)$ and $B\subseteq C^*(Y)$ be Roe-like $C^*$-algebras such that either $B=C_s^*(Y)$ or $UC^*(Y)\subseteq B$. If $A$ embeds onto a hereditary $C^*$-subalgebra of $B$, then all sparse subspaces of $X$ contain no block-rank-one ghost projections in their Roe algebras.
\end{cor}

Finally, we reach the main result of this section.
\begin{prop}\label{prop:rigidity.Roe}
Let $X$ and $Y$ be metric spaces with bounded geometry. Assume that all sparse subspaces of $Y$ contain no block-rank-one ghost projections in their Roe algebras. Then the following are equivalent:
\begin{enumerate}
  \item $X$ is coarsely equivalent to $Y$;
  \item $C^*_u(X)$ is stably $\ast$-isomorphic to $C^*_u(Y)$;
  \item $C^*_s(X)$ is $\ast$-isomorphic to $C^*_s(Y)$;
  \item $UC^*(X)$ is $\ast$-isomorphic to $UC^*(Y)$;
  \item $C^*(X)$ is $\ast$-isomorphic to $C^*(Y)$.
\end{enumerate}
\end{prop}

\begin{proof}
We follow exactly the same proof of \cite[Theorem 1.3]{rigidity_CBC} except we use Corollary~\ref{Cor4.6 in BCL} instead of \cite[Corollary~4.6]{rigidity_CBC} and use Corollary~\ref{cor:3.5 in BCL} instead of \cite[Corollary~3.5]{rigidity_CBC}. Hence, we decide not to repeat the proof word for word.
\end{proof}

\begin{rem}
For a metric space of bounded geometry, it is clear that if all sparse subspaces yield only compact ghost projections in their Roe algebras, then they also do not contain any block-rank-one ghost projection in their Roe algebras. Thus, we have formally generalised \cite[Theorem~1.3]{rigidity_CBC}, and all of the four cases in \cite[Theorem~5.3]{rigidity_CBC} can be included in the setting of Proposition~\ref{prop:rigidity.Roe}.

An advantage of using block-rank-one ghost projections (rather than arbitrary non-compact ghost projections) is that we have a geometric characterisation proved in Corollary~\ref{cor:geometric condition characterisation}, which formulates in terms of measured asymptotic expanders. More precisely, we will show that both conditions concerning Roe algebras and uniform Roe algebras in Definition~\ref{our geometric condition} are equivalent to that $X$ contains no sparse subspaces consisting of ghostly measured asymptotic expanders. Studying measured asymptotic expanders allows us to produce new examples of rigid spaces (see Section~\ref{sec:examples} for details).
\end{rem}

\section{Measured asymptotic expanders}\label{sec:MAE}
In this section, we begin to characterise quasi-locality of a block-rank-one projection by means of measured asymptotic expanders, which were introduced in \cite[Definition~6.1]{dypartI}. The precise statement can be found in Proposition~\ref{prop: measured asymptotic expanders}, which generalises \cite[Theorem~3.11]{Intro} from counting measures to general probability measures. Afterwards, we also establish a structure theorem for measured asymptotic expanders (Corollary~\ref{cor:MAE to ME}), which is our main technical tool to prove Theorem~\ref{thm: main result}.

\subsection{Quasi-locality for block-rank-one projections}
In this subsection, we connect measured asymptotic expanders with the quasi-locality of block-rank-one projections. In \cite{structure,Intro}, we studied quasi-locality of averaging projections, which led to introducing asymptotic expanders, and also showed that an averaging projection is quasi-local \emph{if and only if} it belongs to the uniform Roe algebra (see \cite[Theorem 6.1]{structure}).

Recall that bounded operators which can be approximated by operators with finite propagation are always quasi-local in the following sense:
\begin{defn}\label{defn:quasi-local}
Let $X$ be a discrete metric space, $\H_0$ be a Hilbert space and $T \in \B(\ell^2(X;\H_0))$. We say that $T$ is \emph{quasi-local}\footnote{The notion of quasi-locality was first introduced by Roe in \cite[Part I, Section~5]{Roe88} and \cite[Remark on page 20]{MR1399087}, and we refer readers to \cite{Intro, LWZ19, ST19, SZ18} for more details.} if for any $\varepsilon>0$ there exists some $R>0$ such that for any $A,B \subseteq X$ with $d(A,B) >R$, we have $\|\chi_A T \chi_B\| <\varepsilon$. 
\end{defn}

Throughout this subsection, we fix a sequence of finite metric spaces $\{(X_n,d_n)\}_{n\in \N}$. Let $(X,d)$ be their coarse disjoint union and $\H_0$ be a Hilbert space. Recall from Definition \ref{defn:block rank one projection} that a block-rank-one projection with respect to $\{X_n\}_{n\in \N}$ is an infinite rank projection of the form 
$$P=\bigoplus_{n\in \N} P_n \in \B \left(\bigoplus_{n\in \N}\ell^2(X_n;\H_0)\right)=\B(\ell^2(X;\H_0)),$$
where each $P_n$ is a rank-one projection in $\B(\ell^2(X_n;\H_0))$. Note that each $P_n$ has the form $P_n(\eta)=\langle \eta, \xi_n \rangle \xi_n$ for all $\eta \in \ell^2(X_n;\H_0)$, where $\xi_n$ is a unit vector in $\ell^2(X_n;\H_0)$. In this case, we also say that $P=\bigoplus_{n\in \N} P_n$ is the \emph{block-rank-one projection associated to the unit vectors $\{\xi_n\}_{n\in \N}$}. Then the associated probability measure $m_n$ on $X_n$ is given by $m_n(x)=||\xi_n(x)||^2$.

\begin{rem}
When $\H_0=\C$ and $\xi_n$ is the unit vector $\frac{1}{\sqrt{|X_n|}}\chi_{X_n}\in \ell^2(X_n)$, the corresponding block-rank-one projection $P$ is the so-called \emph{averaging projection}. It has been well-studied in \cite{structure,Intro}, and the most significant result is the fact that $P$ is quasi-local \emph{if and only if} $P$ belongs to the uniform Roe algebra $C^*_u(X)$ \emph{if and only if} $\{X_n\}_{n\in \N}$ forms a sequence of asymptotic expanders in the sense of Definition~\ref{defn:asymptotic expanders} (see \cite[Theorem~6.1]{structure}).
\end{rem}

Similarly to \cite[Lemma~3.8 and Proposition~3.9]{Intro} it is straightforward to obtain the following two results:
\begin{lem}\label{lem: quasi-locality of the projection}
For each $n \in \N$ and any $A,B \subseteq X_n$, we have 
$$\|\chi_A P_n \chi_B\| =\|\chi_A \xi_n\| \cdot \|\chi_B \xi_n\|= \sqrt{m_n(A)\cdot m_n(B)}.$$
\end{lem}

\begin{prop}\label{prop: char for the projection being quasi-local}
A block-rank-one projection $P=\bigoplus_{n\in \N} P_n \in \B(\ell^2(X;\H_0))$ with respect to $\{X_n\}_{n\in \N}$ is quasi-local \emph{if and only if}
$$0 = \lim_{R \to +\infty} \sup\big\{m_n(A)\cdot m_n(B): n\in \N,\  A,B \subseteq X_n,\  d(A,B) \geq R\big\}.$$
\end{prop}
The following observations will allow us to reduce the proof of Theorem~\ref{thm: main result} to the case that all associated probability measures $m_n$ have full support:     
 
For each $n\in \N$, let $Z_n\subseteq X_n$ be the support of $m_n$ and $Z=\bigsqcup_n Z_n$ be the sparse subspace of $X$. Let $Q_n: \ell^2(X_n;\H_0) \to \ell^2(Z_n;\H_0)$ be the orthogonal projection, and
$$Q:=\bigoplus_{n\in \N} Q_n \in \B(\ell^2(X;\H_0), \ell^2(Z;\H_0)).$$
Note that each $Q_nP_nQ_n^* \in \B(\ell^2(Z_n; \H_0))$ is the orthogonal projection onto the one-dimensional subspace spanned by $\xi_n|_{Z_n}$ in $\ell^2(Z_n;\H_0)$, and $Z_n=\supp (\xi_n)$. It follows that $QPQ^*\in \B(\ell^2(Z;\H_0))$ is a block-rank-one projection with respect to $\{Z_n\}_{n\in \N}$. %Applying Proposition~\ref{prop: char for the projection being quasi-local} to both $P$ and $QPQ^*$, 
The following is obvious:
\begin{lem}\label{cor: support quasi-local}
A block-rank-one projection $P\in \B(\ell^2(X;\H_0))$ with respect to $\{X_n\}_{n\in \N}$ is quasi-local \emph{if and only if} $QPQ^*\in \B(\ell^2(Z;\H_0))$ is quasi-local.
\end{lem}

On the other hand, we also have the following observation:
\begin{lem}\label{lem: support Roe}
Let $P\in \B(\ell^2(X;\H_0))$ be a block-rank-one projection with respect to $\{X_n\}_{n\in \N}$. Then the following hold:
\begin{itemize}
\item When $\H_0=\C$, then $P$ belongs to $C^*_u(X)$ \emph{if and only if} $QPQ^*$ belongs to $C^*_u(Z)$. 
\item When $\H_0$ is a separable infinite-dimensional Hilbert space, then $P$ belongs to $C^*(X)$ \emph{if and only if} $QPQ^*$ belongs to $C^*(Z)$.
\end{itemize}
\end{lem}

\begin{proof}
This follows from the fact that $Q$ can be also regarded as an element in $\B(\ell^2(X;\H_0))$ with propagation zero, and $P=Q^*(QPQ^*)Q$. 
\end{proof}

When $\H_0=\C$ and $P$ is the averaging projection with respect to $\{X_n,d_n\}_{n\in \N}$, then we know from \cite[Theorem~B]{Intro} that $P$ is quasi-local if and only if $\{(X_n,d_n)\}_{n\in \N}$ forms a sequence of asymptotic expanders in the following sense: 
 
\begin{defn}[\cite{Intro}]\label{defn:asymptotic expanders}
A sequence of finite metric spaces $\{(X_n,d_n)\}_{n\in \N}$ is called a sequence of \emph{asymptotic expanders}\footnote{In this paper, we do not require the condition $|X_n|\to \infty$ as $n\to \infty$ as a part of the definition.} if for any $\alpha\in (0,\frac{1}{2}]$ there exist $c_\alpha \in (0,1)$ and $R_\alpha>0$ such that for any $n\in \N$ and $A \subseteq X_n$ with $\alpha |X_n| \leq |A| \leq \frac{1}{2}|X_n|$, we have $|\partial_{R_\alpha} A| > c_\alpha|A|$.
\end{defn}

The following result is one of the main motivations for us to study measured asymptotic expanders, and it is also the first step to attack the fundamental theorem (Theorem~\ref{thm: main result}). Since its proof is essentially a repetition of the arguments used to prove \cite[Theorem~3.11]{Intro}, we omit further details.

\begin{prop}\label{prop: measured asymptotic expanders}
Let $(X,d)$ be a coarse disjoint union of a sequence of finite metric spaces $\{(X_n,d_n)\}_{n\in \N}$, and $P=\bigoplus_{n\in \N} P_n \in \B(\ell^2(X))$ be a block-rank-one projection with respect to $\{X_n\}_{n\in \N}$. If $m_n$ are the probability measures associated to $P$, then $P$ is quasi-local \emph{if and only if} $\{(X_n, d_n, m_n)\}_{n\in \N}$ forms a sequence of measured asymptotic expanders.
\end{prop}

For later purposes, we extend the notion of measured asymptotic expanders to general \emph{finite measured metric spaces} defined in Definition~\ref{def:ghostly} as follows:

\begin{defn}[{\cite[Definition~6.1]{dypartI}}]\label{def: measured asymptotic expander}
A sequence of finite measured metric spaces $\{(X_n,d_n,m_n)\}_{n \in \N}$ is called a sequence of \emph{measured asymptotic expanders} if for any $\alpha \in (0,\frac{1}{2}]$ there exist $c_\alpha\in (0,1)$ and $R_\alpha>0$ such that for any $n \in \N$ and $A \subseteq X_n$ with $\alpha \cdot m_n(X_n) \leq m_n(A) \leq \frac{1}{2} m_n(X_n)$, we have $m_n(\partial_{R_\alpha} A) > c_\alpha\cdot m_n(A)$.

In this case, we call functions $\underbar{c}:\alpha \mapsto c_\alpha$ and $\underbar{R}:\alpha \mapsto R_\alpha$ from $(0,\frac{1}{2}]$ to $(0,\infty)$ \emph{parameter functions} of $\{(X_n,d_n,m_n)\}_{n \in \N}$, and $\{(X_n,d_n,m_n)\}_{n\in \N}$ is called a sequence of \emph{measured $(\underbar{c},\underbar{R})$-asymptotic expanders}.
\end{defn}
\begin{rem}
The notion of asymptotic expanders introduced and studied in \cite{structure,Intro} is exactly the notion of \emph{ghostly} measured asymptotic expanders with respect to counting measures (see Remark~\ref{rem:ghostly}).
\end{rem}

\begin{rem}
Recall that in \cite[Theorem~6.16]{dypartI}, the authors established a connection between measured asymptotic expanders and asymptotic expansion in measure for continuous measure-class-preserving actions by means of measured approximating spaces. As a consequence, it provides an efficient and unified way to construct measured asymptotic expanders from strongly ergodic actions (see \cite[Proposition~3.5]{dypartI}).
\end{rem}

We need three auxiliary lemmas in order to prove the structure theorem for measured asymptotic expanders in the next subsection. The first one allows us to handle subsets $A\subseteq X_n$ with $m_n(A) \geq \frac{1}{2} m_n(X_n)$. Although its proof is already given in \cite{dypartI}, we still include it here for the sake of completeness.

\begin{lem}{$($see \cite[Lemma~6.3]{dypartI}$)$}\label{lem:annoying 1/2 bdd constant}
Let $\{(X_n,d_n,m_{n})\}_{n \in \N}$ be a sequence of measured $(\underbar{c},\underbar{R})$-asymptotic expanders. Then for any $\beta \in [\frac{1}{2},1)$, there exist $\tilde c>0$ and $\tilde R>0$ depending only on the parameter functions $\underbar{c}$ and $\underbar{R}$ such that for every $n\in \N$ and $A \subseteq X_n$ with
$\frac{1}{2} m_n(X_n) \leq m_n(A) \leq \beta \cdot m_n(X_n)$, we have $m_n(\partial_{\tilde R} A) > \tilde{c} \cdot m_n(A)$.

Furthermore, if $\underbar{c} \equiv c$ and $\underbar{R} \equiv R$ are constant functions, we can choose $\tilde R=R$ and $\tilde c = \frac{1-\beta}{2\beta}c$.
\end{lem}

\begin{proof}
Replacing $m_n$ by $\frac{1}{m_n(X_n)}m_n$, we may assume that each $m_n$ is a probability measure. Now we fix $\beta \in [\frac{1}{2},1)$ and set $\alpha':=\frac{1-\beta}{2}\in (0,\frac{1}{4}]$. By the hypothesis, there exist $c_{\alpha'}\in (0,1)$ and $R_{\alpha'}>0$ such that for any $n\in \N$ and $A'\subseteq X_n$ with $\alpha'\leq m_n(A') \leq \frac{1}{2}$, we have $m_n(\partial_{R_{\alpha'}} A')>c_{\alpha'}\cdot m_n(A')$.

Given $A\subseteq X_n$ with $\frac{1}{2}\leq m_n(A) \leq \beta$, let $B:= X_n \setminus (\Nd_{R_{\alpha'}}(A))$. Then $m_n(B)\leq m_n(X_n\backslash A)\leq \frac{1}{2}$. If $m_n(B) < \frac{1-\beta}{2}$, then
$m_n(\Nd_{R_{\alpha'}}(A)) > \frac{1+\beta}{2}\geq \frac{1+\beta}{2\beta}m_n(A)$. Since $\Nd_{R_{\alpha'}}(A)=A\sqcup \partial_{R_{\alpha'}}A$, we deduce that $m_n(\partial_{R_{\alpha'}} A) > c_{\alpha'} \frac{1-\beta}{2\beta} m_n(A)$. On the other hand, if $m_n(B)\geq \frac{1-\beta}{2}=\alpha'$ then we have $m_n(\partial_{R_{\alpha'}} B)>c_{\alpha'}\cdot m_n(B)$. Since $\partial_{R_{\alpha'}} B \subseteq \partial_{R_{\alpha'}} A$, we conclude that 
 \[
  m_n(\partial_{R_{\alpha'}} A) \geq m_n(\partial_{R_{\alpha'}} B) >c_{\alpha'} \cdot m_n(B) \geq c_{\alpha'}\frac{1-\beta}{2\beta}m_n(A).
 \]
Thus, we have completed the proof.
\end{proof}

In the next lemma, we show that subspaces of measured asymptotic expanders are themselves measured asymptotic expanders in a uniform way provided that their measures are uniformly bounded below. Although elementary, this lemma plays a crucial role to prove the structure theorem for measured asymptotic expanders, and the corresponding lemma is \emph{not} necessary in the case of asymptotic expanders as in \cite{structure}.

\begin{lem}\label{lem: uniform asymptotic expansion for subspaces}
Let $\{(X_n,d_n,m_n)\}_{n\in \N}$ be a sequence of measured asymptotic expanders.
For any $\beta \in (0,1]$, $\alpha \in (0,\frac{1}{2}]$ and
$c\in (0,1)$ there exists $R>0$ such that for any sequence
$\{Y_n \subseteq X_n\}_{n \in \N}$ with $m_n(Y_n) \geq \beta m_n(X_n)$ and any
$A \subseteq Y_n$ with $\alpha \cdot m_n(Y_n) \leq m_n(A) \leq \frac{1}{2} m_n(Y_n)$,
we have $m_n(\partial_R^{Y_n} A) > c\cdot m_n(A)$.
\end{lem}

\begin{proof}
Replacing $m_n$ by $\frac{1}{m_n(X_n)}m_n$, we may assume that each $m_n$ is a probability measure. From Proposition~\ref{prop: char for the projection being quasi-local} and Proposition~\ref{prop: measured asymptotic expanders}, we know that for any $\varepsilon>0$, there exists some $R_\varepsilon>0$ such that
for any $n \in \N$ and $A,B \subseteq X_n$ with $d(A,B) \geq R_\varepsilon$
we have $m_n(A)\cdot m_n(B) < \varepsilon$. In particular, it holds for $\varepsilon:=\frac{\alpha\beta^2(1-c)}{2}>0$.

Fix an arbitrary sequence
$\{Y_n \subseteq X_n\}_{n \in \N}$ with $m_n(Y_n) \geq \beta$, and any
$A \subseteq Y_n$ with
$\alpha \cdot m_n(Y_n) \leq m_n(A) \leq \frac{1}{2}\cdot m_n(Y_n)$.
Suppose that $m_n(\partial^{Y_n}_{R_\varepsilon} A) \leq c\cdot m_n(A)$. Then
due to the decomposition
$$Y_n= (Y_n \setminus \Nd^{Y_n}_{R_\varepsilon}(A)) \sqcup A \sqcup \partial^{Y_n}_{R_\varepsilon} A,$$
we have that
\begin{eqnarray*}
m_n(Y_n \setminus \Nd^{Y_n}_{R_\varepsilon}(A)) =  m_n(Y_n) - m_n(A) - m_n(\partial^{Y_n}_{R_\varepsilon}A) \geq  \frac{1-c}{2}m_n(Y_n) \geq  \frac{1-c}{2} \beta.
\end{eqnarray*}
Hence, it follows that
$$m_n(A) \cdot m_n(Y_n \setminus \Nd^{Y_n}_{R_\varepsilon}(A)) \geq \alpha \cdot \beta \cdot \frac{1-c}{2} \beta=\varepsilon.$$
Since $d(A, Y_n \setminus \Nd^{Y_n}_{R_\varepsilon}(A))\geq R_\epsilon$, we have reached a contradiction.
\end{proof}

We finish this subsection by showing the existence of subspaces whose measure is more “balanced” in the lemma below. Note that every counting measure is automatically “balanced” so that the next lemma is redundant in the study of asymptotic expanders. However, it is crucial for us to be able to apply Proposition~\ref{prop: spetral gap for measured expanders} in proving our fundamental theorem (Theorem~\ref{thm: main result}). 

In order to formulate the lemma in a concise way, we here introduce the following notation: \emph{given two non-negative numbers $a,b \in [0,\infty)$, we denote $a \thicksim_s b$ for some $s\in (0,1)$ if $sa \leq b \leq \frac{a}{s}$.} 

\begin{lem}\label{lem: bounded ratio}
Let $\{(X_n,d_n,m_n)\}_{n\in \N}$ be a sequence of finite measured metric spaces with uniformly bounded
geometry. Given $R>0$ and $s\in (0,1)$,
there exists a subspace $X^{R,s}_n \subseteq X_n$ for each $n\in \N$ satisfying
the following:
\begin{enumerate}
  \item For any $x,y\in X_n^{R,s}$ with $d_n(x,y) \leq R$, we have
    $m_n(x) \thicksim_s m_n(y)$; \label{lem:bounded-ratio-1}
  \item $m_n(X_n \setminus X_n^{R,s}) \leq s\cdot N_R\cdot m_n(X^{R,s}_n)$, where $N_R:=\sup_{n\in \N}\sup_{x\in X_n}|B(x,R)|$.
    \label{lem:bounded-ratio-2}
\end{enumerate}
\end{lem}

\begin{proof}
We will recursively construct the desired $X^{R,s}_n \subseteq X_n$ for each $n\in \N$. Fix an $n\in \N$, and choose any $x_0 \in X_n$ such that
$m_n(x_0)=\max_{x\in X_n} m_n(x)$. Define
$$
X_{n,1}:=X_n\backslash \{y\in B(x_0,R):m_n(y)<s\cdot m_n(x_0)\}.
$$
Since $s<1$, we see that $x_0\in X_{n,1}$. Thereafter we choose any $x_1 \in X_{n,1} \setminus \{x_0\}$ so that
$m_n(x_1)=\max_{x\in X_{n,1} \setminus \{x_0\}} m_n(x)$, and define
$$
X_{n,2}:=X_{n,1}\backslash \{y\in X_{n,1} \cap B(x_1,R):m_n(y) <s\cdot m_n(x_1)\}.$$
Recursively, we choose $x_k \in X_{n,k} \setminus \{x_{k-1},\dots,x_0\}$ so that $m_n(x_k)=\max_{x\in X_{n,k} \setminus \{x_{k-1},\dots,x_0\}} m_n(x)$, and define
$$
X_{n,k+1}:=X_{n,k}\backslash \{y\in X_{n,k} \cap B(x_k,R):m_n(y) <s\cdot m_n(x_k)\}.$$
By convention, $X_{n,0}:=X_n$.
Since $X_n$ is finite, the recursive process will finish after finitely many steps, when  $X_{n,k_0}=\{x_0,x_1,\ldots, x_{k_0}\}$ for some $k_0\in \N\cup\{0\}$. Note that $m_n(x_0)\geq m_n(x_1)\geq \ldots \geq m_n(x_{k_0})$. We denote the resulting subspace by $X_n^{R,s}=X_{n,k_0}$. Clearly, $X_n^{R,s}$ satisfies condition~\eqref{lem:bounded-ratio-1}.

Let $A_i:=\{y\in X_{n,i} \cap B(x_i,R):m_n(y) <s\cdot m_n(x_i)\}$ for $i=0,1,\ldots,k_0$. Then $$
X_n \setminus X_n^{R,s}=A_{k_0}\sqcup A_{k_0-1}\sqcup \ldots \sqcup A_0.
$$
Since $m_n(A_i)\leq s\cdot N_R\cdot m_n(x_i)$ and $X_n^{R,s}=\{x_0,x_1,\ldots, x_{k_0}\}$, it is easily seen that $X_n^{R,s}$ also satisfies condition~\eqref{lem:bounded-ratio-2}, as required.
\end{proof}

\subsection{Structure theorem for measured asymptotic expanders}\label{Structure Theorem}
Recall that the authors in \cite[Theorem~3.7]{structure} proved a structure theorem for asymptotic expanders. Here we extend the structure result to the case of measured asymptotic expanders by a slightly different and improved argument.

\begin{thm}\label{thm:structure thm}
Let $\{(X_n,d_n,m_n)\}_{n \in \N}$ be a sequence of finite measured metric spaces with uniformly bounded
geometry. Then the following are equivalent:
\begin{enumerate}
  \item $\{(X_n,d_n,m_{n})\}_{n \in \N}$ is a sequence of measured asymptotic expanders;
  \item for any $c\in(0,1)$, there exists a sequence $\{\alpha_k\}_{k \in \N}$ in $(0,1)$ with $\alpha_k\to 0$, and a positive sequence $\{R_k\}_{k\in \N}$ such that for each fixed $n \in \N$ there exists a
    sequence of subspaces $\{Y_{n,k}\}_{k\in \N}$ in $X_n$ satisfying the following for every $k\in \N$:
  \begin{itemize}
    \item[(i)] $Y_{n,k} \subseteq \supp (m_n)$ and $m_n(Y_{n,k}) \geq (1-\alpha_k)\cdot m_n(X_n)$;
    \item[(ii)] for each $A \subseteq Y_{n,k}$ with $0<m_n(A) \leq \frac{1}{2}m_n(Y_{n,k})$, then $m_n(\partial_{R_k}^{Y_{n,k}} A) > c\cdot m_n(A)$;
    \item[(iii)] there exists $s_k\in (0,1)$ such that for any $x,y \in Y_{n,k}$ with $d_n(x,y) \leq R_k$, we have $m_n(x) \thicksim_{s_k} m_n(y)$.
  \end{itemize}
\end{enumerate}
The hypothesis of $(iii)$ is \emph{not} needed for the implication ``$(2)\Rightarrow (1)$''.
\end{thm}
\begin{proof}
\emph{``(1) $\Rightarrow$ (2)''}:
Fix any $c\in (0,1)$. For each $\alpha \in (0,\frac{1}{2}]$, let $R_\alpha$ satisfy
the conclusion in Lemma \ref{lem: uniform asymptotic expansion for subspaces}
for $\beta=\frac{1}{2}$. Without loss of generality, we may assume that
$R_\alpha > 1$ and $\lim_{\alpha\to 0}R_\alpha=\infty$. Let $N_{R_\alpha}:=\sup_{n\in \N}\sup_{x\in X_n}|B(x,R_\alpha)|$ so that $1\leq N_{R_\alpha}<\infty$. If we set $s_\alpha:=\frac{1}{R_\alpha\cdot N_{R_\alpha}} \in (0,1)$, then Lemma~\ref{lem: bounded ratio} provides a subspace $X_{n,\alpha}:=X_n^{R_\alpha,s_\alpha}$ in $X_n$ for each
$n\in\N$ satisfying the following:
\begin{itemize}
  \item For any $x,y\in X_{n,\alpha}$ with $d_n(x,y) \leq R_\alpha$, we have $m_n(x) \thicksim_{s_\alpha} m_n(y)$;
  \item $m_n(X_n \setminus X_{n,\alpha}) \leq s_\alpha\cdot N_{R_\alpha}\cdot m_n(X_{n,\alpha}) = \frac{1}{R_\alpha} \cdot m_n(X_{n,\alpha})$.
\end{itemize}
From the second condition, we obtain that
$$m_n(X_{n,\alpha}) \geq \frac{R_\alpha}{R_\alpha+1}m_n(X_n) \geq \frac{1}{2}m_n(X_n)\ \text{for every $n\in \N$.}$$
So we know from Lemma~\ref{lem: uniform asymptotic expansion for subspaces} that for any $A \subseteq X_{n,\alpha}$ with $\alpha \cdot m_n(X_{n,\alpha}) \leq m_n(A) \leq \frac{1}{2} m_n(X_{n,\alpha})$, we have $m_n(\partial_{R_\alpha}^{X_{n,\alpha}} A) > c \cdot m_n(A)$.

Now we assume that $\alpha \in (0,\frac{c}{4+2c}]$. For each $n\in \N$, we consider a family $\mathcal{F}_{n,\alpha}$ of subsets in $X_{n,\alpha}$ given by
$$\mathcal{F}_{n,\alpha}:=\big\{A \subseteq X_{n,\alpha}: m_n(A) \leq \frac{1}{2} m_n(X_{n,\alpha}) \mbox{~and~} m_n(\partial^{X_{n,\alpha}}_{R_\alpha} A) \leq c\cdot m_n(A)\big\}.$$
We see that $\mathcal{F}_{n,\alpha}$ always contains the empty set $\emptyset$ so that it admits a maximal element directed by the inclusion. Let $F_{n,\alpha}$ be a maximal element in $\mathcal{F}_{n,\alpha}$, and set $Y_{n,\alpha}:=X_{n,\alpha} \setminus F_{n,\alpha}$. By construction, we have $m_n(F_{n,\alpha}) < \alpha \cdot m_n(X_{n,\alpha})$ for each $n\in\N$. Thus, it follows that
$$m_n(Y_{n,\alpha}) > (1-\alpha)m_n(X_{n,\alpha}) \geq \frac{(1-\alpha)\cdot R_\alpha}{R_\alpha+1}m_n(X_n)\ \text{for every $n\in\N$.}$$
In the following, we divide into two cases:

\emph{Case I.} Let $A \subseteq Y_{n,\alpha}$ satisfy $0<m_n(A) \leq \frac{1}{2} m_n(X_{n,\alpha}) - m_n(F_{n,\alpha})$. In particular, $
m_n(A)\leq \frac{1}{2} m_n(Y_{n,\alpha})$ and $m_n(A \sqcup F_{n,\alpha}) \leq \frac{1}{2} m_n(X_{n,\alpha})$. Since $\partial^{X_{n,\alpha}}_{R_\alpha} (A \sqcup F_{n,\alpha}) \subseteq (\partial^{X_{n,\alpha}}_{R_\alpha} F_{n,\alpha}) \cup (\partial^{Y_{n,\alpha}}_{R_\alpha} A)$, we deduce that
$$m_n\big(\partial^{X_{n,\alpha}}_{R_\alpha} (A \sqcup F_{n,\alpha})\big) \leq m_n\big( \partial^{X_{n,\alpha}}_{R_\alpha} F_{n,\alpha} \big) + m_n \big( \partial^{Y_{n,\alpha}}_{R_\alpha} A  \big) \leq c\cdot m_n(F_{n,\alpha}) + m_n \big( \partial^{Y_{n,\alpha}}_{R_\alpha} A  \big).$$
On the other hand, since $F_{n,\alpha}$ is maximal in $\mathcal{F}_{n,\alpha}$ and $A\neq \emptyset$ we have that 
$$m_n\big(\partial^{X_{n,\alpha}}_{R_\alpha} (A \sqcup F_{n,\alpha})\big) > c \cdot m_n(A \sqcup F_{n,\alpha}) = c \cdot m_n(A) + c \cdot m_n(F_{n,\alpha}).$$
Combining them together, we conclude that $m_n ( \partial^{Y_{n,\alpha}}_{R_\alpha} A ) > c \cdot m_n(A)$.

\emph{Case II.} Let $A \subseteq Y_{n,\alpha}$ satisfy $\frac{1}{2} m_n(X_{n,\alpha}) - m_n(F_{n,\alpha}) < m_n(A) \leq \frac{1}{2} m_n(Y_{n,\alpha})$. Since $\alpha \leq \frac{1}{4}$ and $m_n(F_{n,\alpha}) < \alpha \cdot m_n(X_{n,\alpha})$, we have that
$$\frac{1}{2} m_n(X_{n,\alpha})\geq  m_n(A) > \frac{1}{2} m_n(X_{n,\alpha}) - m_n(F_{n,\alpha}) > \big(\frac{1}{2}-\alpha\big)\cdot m_n(X_{n,\alpha}) \geq \alpha\cdot m_n(X_{n,\alpha}).$$
It follows that $m_n (\partial^{X_{n,\alpha}}_{R_\alpha} A) > c\cdot m_n(A)$. As $\partial^{X_{n,\alpha}}_{R_\alpha} A\subseteq (\partial^{Y_{n,\alpha}}_{R_\alpha} A)\sqcup F_{n,\alpha}$, we see that
$$m_n (\partial^{Y_{n,\alpha}}_{R_\alpha} A) \geq m_n (\partial^{X_{n,\alpha}}_{R_\alpha} A) - m_n(F_{n,\alpha}) > c\cdot m_n(A) - m_n(F_{n,\alpha}).$$
Combining the facts that $\alpha \leq \frac{c}{4+2c}$, $m_n(A) > \frac{1}{2} m_n(X_{n,\alpha}) - m_n(F_{n,\alpha})$ and $m_n(F_{n,\alpha}) < \alpha\cdot  m_n(X_{n,\alpha})$, we deduce that
$$m_n(A) > \frac{1}{c+2} \cdot m_n(X_{n,\alpha}),$$
which further implies that 
$$c\cdot m_n(A) - m_n(F_{n,\alpha}) > \frac{c}{2}\cdot m_n(A).$$
Hence, we obtain $m_n ( \partial^{Y_{n,\alpha}}_{R_\alpha} A ) > \frac{c}{2}\cdot m_n(A)$ in this case.

In conclusion, for a given $c\in (0,1)$ and any $\alpha\in (0,\frac{c}{4+2c}]$ there exist $R_\alpha>1$ with $\lim_{\alpha\to 0}R_\alpha=\infty$, $s_\alpha \in (0,1)$, and a sequence of subspaces $\{Y_{n,\alpha}\subseteq X_n\}_{n\in\N}$ such that for all $n\in\N$ and $\alpha\in(0,\frac{c}{4+2c}]$:
\begin{itemize}
\item $m_n(Y_{n,\alpha}) > \frac{(1-\alpha)\cdot R_\alpha}{R_\alpha+1}m_n(X_n)$;
\item for each $A \subseteq Y_{n,\alpha}$ with $0< m_n(A) \leq \frac{1}{2} m_n(Y_{n,\alpha})$ we have $m_n(\partial_{R_\alpha}^{Y_{n,\alpha}} A) > \frac{c}{2}\cdot  m_n(A)$;
\item for any $x,y \in Y_{n,\alpha}$ with $d_n(x,y) \leq R_\alpha$ we have $m_n(x) \thicksim_{s_\alpha} m_n(y)$.
\end{itemize}
Now we choose a sequence $\{\tilde{\alpha}_k\}_{k\in \N}$ in $(0,\frac{c}{4+2c}]$ such that $\tilde{\alpha}_k\to 0$, and set $R_k:=R_{\tilde{\alpha}_k}$,
$s_k:=s_{\tilde{\alpha}_k}$,
and $Y_{n,k}:=Y_{n,\tilde{\alpha}_k} \cap \supp(m_n)$. Since $\lim_{k\to \infty}\frac{(1-\tilde{\alpha}_k)R_k}{R_k+1} =1$, we complete the proof by letting $\alpha_k:=1-\frac{(1-\tilde{\alpha}_k)R_k}{R_k+1}$.

~\

\emph{``(2) $\Rightarrow$ (1)''}: We assume that condition (2) holds with the constants therein. Given $\alpha\in (0,\frac{1}{2}]$, we take a $k\in \N$ such
that $\alpha_k\leq \frac{\alpha}{8}$. For any $n \in \N$ and
$A \subseteq X_n$ with $\alpha \cdot m_n(X_n) \leq m_n(A) \leq \frac{1}{2}m_n(X_n)$, we observe that
\[
m_n(A\cap Y_{n,k}) \geq m_n(A) - m_n(X_n \setminus Y_{n,k}) \geq m_n(A) - \alpha_k \cdot m_n(X_n) \geq m_n(A) - \frac{\alpha}{2}\cdot m_n(X_n) \geq \frac{1}{2}m_n(A).
\]
In the following, we divide into two cases:

\emph{Case I.} When $m_n(A \cap Y_{n,k}) \leq \frac{1}{2}m_n(Y_{n,k})$: By the hypothesis of $(ii)$, we obtain that 
\[
  m_n(\partial^{X_n}_{R_k} A) \geq m_n(\partial^{Y_{n,k}}_{R_k} (A \cap Y_{n,k})) > c\cdot m_n(A \cap Y_{n,k}) \geq \frac{c}2\cdot m_{n}(A).
\]

\emph{Case II.} When $m_n(A \cap Y_{n,k}) > \frac{1}{2}m_n(Y_{n,k})$: Since
$m_n(Y_{n,k}) \geq (1-\alpha_k) \cdot m_n(X_n)\geq \frac{7}{8}m_n(X_n)$, we have that
$$m_n(A \cap Y_{n,k}) \leq m_n(A) \leq \frac{1}{2}m_n(X_n) \leq \frac{4}{7}m_n(Y_{n,k}).$$
As $\{(Y_{n,k},d_n,m_n)\}_{n\in \N}$ forms a sequence of measured asymptotic expanders for each fixed $k\in \N$, we can apply Lemma~\ref{lem:annoying 1/2 bdd constant} to obtain two positive constants
$c'$ and $R_k'$ only depending on $c$ and $R_k$ such that for any $k\in \N$ and $B\subseteq Y_{n,k}$ with $\frac{1}{2}m_n(Y_{n,k})\leq m_n(B) \leq \frac{4}{7}m_n(Y_{n,k})$, we have
$m_n(\partial^{Y_{n,k}}_{R_k'} B) > c'\cdot m_n(B)$. Hence, we obtain
$$m_n(\partial^{X_n}_{R_k'} A) \geq m_n\big(\partial^{Y_{n,k}}_{R_k'} (A \cap Y_{n,k}) \big)> c'\cdot m_n(A \cap Y_{n,k}) \geq \frac{c'}{2}\cdot m_n(A).$$

Combining these two cases, we conclude that $\{(X_n,d_n,m_{n})\}_{n \in \N}$ is a sequence of measured asymptotic expanders, as desired. 
\end{proof}

\begin{rem}
Note that the assumption of uniformly bounded geometry in Theorem~\ref{thm:structure thm} is only needed to guarantee the condition $(iii)$, which is automatically true when all $m_n$ are counting measures. This is the reason why \cite[Theorem~3.7]{structure} does not require the assumption of uniformly bounded geometry. Hence, Theorem~\ref{thm:structure thm} completely recovers \cite[Theorem~3.7]{structure}.
\end{rem}

We end this subsection by showing that a sequence of measured asymptotic expanders always admits a “uniform exhaustion” by measured expanders with bounded ratios in measure as follows:

Let $\{(X_n,d_n,m_n)\}_{n\in \N}$ be a sequence of measured asymptotic expanders with uniformly bounded geometry. For each fixed $k\in \N$, we consider a sequence of subspaces $\{Y_{n,k}\}_{n\in \N}$ satisfying the condition (2) in Theorem~\ref{thm:structure thm}. We will make $\{Y_{n,k}\}_{n\in \N}$ into a sequence of finite measured graphs in the following sense:

\begin{defn}
A sequence of finite measured metric spaces $\{(V_n,E_n,m_n)\}_{n\in \N}$ is called a sequence of \emph{finite measured graphs} if each $(V_n,E_n)$ is a finite (connected and undirected) graph equipped with the edge-path metric. 
 \end{defn} 
More concretely, we define the vertex set $V_{n,k}:= Y_{n,k}$, the edge set $E_{n,k}$ by $u\thicksim_{E_{n,k}} v$ if and only if $u,v\in V_{n,k}$ satisfy $0< d_n(u,v) \leq R_k$, and the measure $m_{n,k}:=m_n |_{V_{n,k}}$. Then $Y_{n,k} \subseteq \supp( m_n)$ and condition (ii) in Theorem~\ref{thm:structure thm} implies that $m_{n,k}$ has full support and $(V_{n,k},E_{n,k})$ is a finite (connected\footnote{Just as for expanders, the expansion condition \ref{thm:structure thm}(2)(ii) forces any non-empty subset to have a non-empty boundary.} and undirected) graph. Hence, $\{(V_{n,k}, E_{n,k}, m_{n,k})\}_{n\in \N}$ is a sequence of finite measured graphs. If we denote the edge-path metric on $V_{n,k}$ by $d_{n,k}$, then $d_n(u,v) \leq R_k \cdot d_{n,k}(u,v)$ for all $u,v \in V_{n,k}$ and $\partial^{V_{n,k}} A=\partial^{Y_{n,k}}_{R_k} A$ for $A\subseteq V_{n,k}$. In particular, $\{(V_{n,k}, E_{n,k})\}_{n\in \N}$ has uniformly bounded valency and the inclusion map $i_{n,k}:(V_{n,k},d_{n,k}) \to (X_n,d_n)$ is $R_k$-Lipschitz.
  
Due to the above construction and Theorem~\ref{thm:structure thm} (2), we have the following:
\begin{enumerate}
\item[(i)] $m_{n,k}$ has full support and $m_{n,k}(V_{n,k}) \geq (1-\alpha_k)\cdot m_n(X_n)$;
\item[(ii)] for each $A \subseteq V_{n,k}$ with $0<m_{n,k}(A) \leq \frac{1}{2}m_{n,k}(V_{n,k})$, then $m_{n,k}(\partial^{V_{n,k}} A) > c \cdot m_{n,k}(A)$;
\item[(iii)] there exists $s_k\in (0,1)$ such that for any adjacent pair of vertices $u\thicksim_{E_{n,k}} v$, we have $m_{n,k}(u) \thicksim_{s_k} m_{n,k}(v)$.
\end{enumerate}

The following definition is derived from condition (ii) above and it is a measured version of the classical notion of expanders (see Definition~\ref{defn:expanders} and we also refer the reader to \cite{measured_II} for more details about measured expanders):

\begin{defn}\label{def: measured expanders}
Let $\{(V_n,E_n,m_n)\}_{n \in \N}$ be a sequence of finite measured graphs. We say that $\{(V_n,E_n,m_n)\}_{n \in \N}$ is a sequence of \emph{measured expanders}\footnote{In the classical definition of expanders, the sequence of graphs is usually required to have $\lim_{n\to \infty}|V_n| =\infty$ and uniformly bounded valency. Here we take the liberty of excluding these restrictions for a greater generality.} if there exists $c>0$ such that for any $n \in \N$ and $A \subseteq V_n$ with $0< m_n(A) \leq \frac{1}{2} m_n(V_n)$, we have $m_n(\partial^{V_n} A) > c\cdot m_n(A)$. In this case, we also say that $\{(V_n,E_n,m_n)\}_{n \in \N}$ is a sequence of \emph{$c$-measured expanders} (a.k.a. measured $(c,1)$-asymptotic expanders, cf. Definition \ref{def: measured asymptotic expander}).
\end{defn}

\begin{rem}\label{rem:measured-expanders-are-connected}
In Definition~\ref{def: measured expanders}, observe that if all of $m_n$ have full support then expansion in measure will force all graphs to be connected.
\end{rem}

Before we state our structure theorem for measured asymptotic expanders by means of measured expanders, let us record the following lemma which is needed in Section~\ref{sec:examples}. 
\begin{lem}\label{lem:annoying 1/2 bdd expander}
Let $\{(V_n,E_n,m_n)\}_{n \in \N}$ be a sequence of $c$-measured expanders for some $c>0$. Then for any $\beta \in (0,1)$, there exists some $c_\beta>0$ depending only on $c$ and $\beta$ such that for any $n \in \N$ and $A \subseteq V_n$ with $0< m_n(A) \leq \beta\cdot m_n(V_n)$, we have $m_n(\partial^{V_n} A) > c_\beta\cdot m_n(A)$.
\end{lem}

\begin{proof}
If $0<\beta\leq \frac{1}{2}$, the conclusion follows from the definition of $c$-measured expanders.

If $\frac{1}{2}<\beta <1$ and $m_n(A)\leq \frac{1}{2}m_n(V_n)$, it again follows from the definition of $c$-measured expanders. If $\frac{1}{2}<\beta <1$ and $m_n(A)> \frac{1}{2}m_n(V_n)$, it follows directly from Lemma~\ref{lem:annoying 1/2 bdd constant} for $\underbar{c} \equiv c$ and $\underbar{R} \equiv 1$.
\end{proof}

We are ready to prove the following corollary:
\begin{cor}\label{cor:MAE to ME}
Let $\{(X_n,d_n,m_n)\}_{n \in \N}$ be a sequence of finite measured metric spaces with uniformly bounded
geometry. Then the following are equivalent:
\begin{enumerate}
  \item $\{(X_n,d_n,m_{n})\}_{n \in \N}$ is a sequence of measured asymptotic expanders;
  \item there exist $c>0$, a sequence $\{\alpha_k\}_{k \in \N}$ in $(0,1)$ with $\alpha_k\to 0$, a sequence $\{s_k\}_{k\in \N}$ in $(0,1)$, and a positive sequence $\{R_k\}_{k\in \N}$ such that for any $n, k \in \N$ there exist a finite graph $(V_{n,k},E_{n,k})$ and a $R_k$-Lipschitz injective map $i_{n,k}:V_{n,k} \to X_n$ satisfying the following:
  \begin{itemize} 
    \item[(i)] the pullback measure\footnote{As $i_{n,k}$ is injective, the pullback measure is well-defined by the formula $i_{n,k}^*(m_n)(A)=m_n(i_{n,k}(A))$ for any $A\subseteq V_{n,k}$.} $m_{n,k}:=i_{n,k}^*(m_n)$ on $V_{n,k}$ has full support and $m_{n,k}(V_{n,k}) \geq (1-\alpha_k)\cdot m_n(X_n)$;
    \item[(ii)] for each $k\in \N$, $\{(V_{n,k},E_{n,k},m_{n,k})\}_{n\in \N}$ is a sequence of $c$-measured expanders with uniformly bounded valency;
    \item[(iii)] for any adjacent pair of vertices $u\thicksim_{E_{n,k}} v$, we have $m_{n,k}(u) \thicksim_{s_k} m_{n,k}(v)$.\footnote{Recall that given two positive numbers $a$ and $b$, we denote $a \thicksim_s b$ for some $s\in (0,1)$ if $sa \leq b \leq \frac{a}{s}$.}
      \end{itemize}
  \end{enumerate}
  The hypothesis of $(iii)$ is \emph{not} needed for the implication ``$(2)\Rightarrow (1)$''.
\end{cor}

\begin{proof}
``$(1)\Rightarrow (2)$'': It follows exactly from our preceding construction. 

``$(2)\Rightarrow (1)$'': For each fixed $n\in \N$, let $Y_{n,k}:=i_{n,k}(V_{n,k})$ for every $k\in \N$. Since $
i_{n,k}(\partial^{V_{n,k}} A)\subseteq \partial^{Y_{n,k}}_{R_k} (i_{n,k}(A))
$ for any $A\subseteq V_{n,k}$, it follows that $\{Y_{n,k}\}_{k\in \N}$ is a sequence of subspaces in $X_n$ satisfying $(i)$ and $(ii)$ in Theorem~\ref{thm:structure thm} as desired.
\end{proof}

As a direct consequence of Corollary~\ref{cor:MAE to ME}, we obtain the following:
\begin{cor}\label{cor:measured weak embedding}
A metric space $X$ does not measured weakly contain any measured expanders with uniformly bounded valency \emph{if and only if} it does not measured weakly contain any measured asymptotic expanders with uniformly bounded geometry.
\end{cor}

\section{The Poincar\'{e} inequality and spectral gaps}

The aim of this section is to study the $L^p$-Poincar\'{e} inequality and spectral gaps for measured expanders with bounded measure ratios. To this end, we start by recalling the case of reversible random walks as established in \cite{measured_II}.

\subsection{Reversible random walks}\label{sec: random walks}
This subsection is devoted to recalling the $L^p$-Poincar\'{e} inequality and spectral gaps for reversible random walks. For more details about the theory of random walks, we refer to the textbooks \cite{BdlHV:2008,woess2000random}.

A \emph{random walk} or a \emph{Markov kernel} on a non-empty set $V$ is a map $r: V \times V \longrightarrow [0,\infty)$ such that $\sum_{v \in V} r(u,v) = 1$ for any $u\in V$. A \emph{stationary measure} $\mu$ for a random walk $r$ is a function $\mu: V \longrightarrow (0,\infty)$ such that $\mu(u)r(u,v)=\mu(v)r(v,u)$ for any $u,v\in V$. A random walk is called \emph{reversible} if it admits at least one stationary measure. In the reversible case, the map $a: V \times V \longrightarrow [0,\infty)$ defined by $a(u,v):=\mu(u)r(u,v)$ is called the \emph{conductance} function. Clearly, $a$ is \emph{symmetric} in the sense that $a(u,v)=a(v,u)$ for all $u,v\in V$, and we also have $\mu(u)=\sum_{v\in V}a(u,v)$ for all $u\in V$. Conversely, let $a: V\times V \to [0,\infty)$ be a symmetric map such that $\mu(u):=\sum_{v\in V} a(u,v)$ is positive and finite for each $u \in V$. Then the formula $r(u,v):=\frac{a(u,v)}{\mu(u)}$ defines a reversible random walk on $V$ with stationary measure $\mu$.

Given a reversible random walk $r$ on a non-empty set $V$ with a stationary measure $\mu$, we can endow $V$ with a (not necessarily connected\footnote{The constructed graph $(V,E)$ is connected if and only if the random walk $r$ is irreducible (see \cite[Example~5.1.1]{BdlHV:2008} for details).} but undirected) graph structure $(V,E)$ by requiring that $u\thicksim_E v$ is an edge if and only if $r(u,v)>0$ (which is also equivalent to that $r(v,u)>0$). Since the corresponding conductance function $a$ is symmetric, we define $a(e):=a(u,v)=a(v,u)$ for each edge $e\in E$ connecting vertices $u$ and $v$.
For $D \subseteq E$, its \emph{area} is defined to be $a(D):= \sum_{e\in D} a(e)$ and $a(\emptyset)=0$. Now let $m$ be an arbitrary measure on $V$ with full support such that $(V,E,m)$ forms a finite measured graph. Then we define the \emph{$(\mu,a,m)$-Cheeger constant} of $(V,E,m)$ to be
\[
\min\big\{\frac{a(\partial^E A)}{\mu(A)}: A\subseteq V \mbox{~with~} 0<m(A) \leq \frac{1}{2}m(V)\big\},
\]
where $\partial^E A$ is the edge boundary of $A$ (\emph{i.e.}, the set of edges with exactly one endpoint in $A$). Similarly to Remark~\ref{rem:measured-expanders-are-connected}, if the $(\mu,a,m)$-Cheeger constant is positive, then the graph $(V,E)$ is automatically connected.

Consider the following Hilbert space:
$$\ell^2(V;\mu):=\big\{f: V \to \C ~\big|~ \sum_{v\in V} |f(v)|^2\mu(v)< \infty\big\} \quad \mbox{with} \quad \langle f_1, f_2\rangle_\mu:=\sum_{v\in V} f_1(v)\overline{f_2(v)} \mu(v).$$
The \emph{graph Laplacian} $\Delta \in \B(\ell^2(V;\mu))$ associated to the reversible random walk $r$ is defined as
\begin{equation}\label{EQ10}
(\Delta f)(v):=f(v)-\sum_{u\in V:u\thicksim_E v}f(u)r(v,u)
\end{equation}
for $f\in \ell^2(V;\mu)$ and $u,v \in V$. In fact, the graph Laplacian $\Delta$ is a positive bounded operator with
norm at most $2$. When the constructed graph $(V,E)$ is connected, $\Delta f=0$ if and only if $f$ is constant (see \emph{e.g.}, \cite[Proposition~5.2.2]{BdlHV:2008} for details).

We end this subsection by stating the $L^p$-Poincar\'{e} inequality and the spectral gap of the graph Laplacian for reversible random walks as established in \cite{measured_II} (we refer the reader to \cite{alon1986eigenvalues,alon1985lambda1,dodziuk1984difference,Mat97,tanner1984explicit} for results of this type for classical expanders).  
\begin{prop}\cite[Proposition~3.2 and Proposition~3.8]{measured_II}\label{prop: Cheeger to Poincare}
Let $r$ be a reversible random walk on a non-empty and finite set $V$ with a stationary measure $\mu$ such that $a$ is the associated conductance function and $(V,E)$ is the associated graph structure. 

If $m$ is any non-trivial and finite measure on $V$ of full support such that the $(\mu,a,m)$-Cheeger constant $c$ is positive, then the following hold:
\begin{itemize}
\item[(1)] The spectrum of the graph Laplacian $\Delta \in \B(\ell^2(V;\mu))$ is contained in $\{0\} \cup [c^2/2, 2]$;
\item[(2)] For any $p\in [1,\infty)$, there exists a positive constant $c_p$ only depending on $c$ and $p$ such that for any map $f:V \to \C$ we have the following $L^p$-\emph{Poincar\'{e} inequality}:
\begin{equation}\label{EQ6}
\sum_{u,v\in V:u\thicksim_E v}|f(u)-f(v)|^p a(u,v) \geq c_p \sum_{u,v\in V} |f(u)-f(v)|^p \frac{\mu(u)\mu(v)}{\mu(V)}.
\end{equation}
\end{itemize}
\end{prop}

\subsection{Spectral projections for measured expanders}\label{subsection: avgeraging projection of measured expanders}

We now return to the general case of measured expanders and explore analogous $L^p$-Poincar\'{e} inequality and spectral gaps.

However, measures involved in the definition of measured expanders might not come from reversible random walks in general so that we cannot directly apply Proposition~\ref{prop: Cheeger to Poincare}. To overcome this issue, we build auxiliary random walks whose stationary measures can uniformly control the original measures in measured expanders with bounded measure ratios on adjacent vertices. More precisely, we need the following key lemma:
\begin{lem}\label{lem:measured expanders to reversible ones}
Let $\{(V_n,E_n,m_n)\}_{n \in \N}$ be a sequence of $c$-measured expanders for some $c>0$ with valency uniformly bounded by $K\geq 1$. Assume that each $m_n$ has full support and that there exists $s\in (0,1)$ such that $m_n(u) \thicksim_s m_n(v)$ for any edge $u \thicksim_{E_n} v$ and any $n\in \N$. 

Then for every $n\in \N$ there exists a reversible random walk $r_n: V_n \times V_n \to [0,\infty)$ such that the associated stationary measure $\mu_n: V_n \to (0,\infty)$ and the associated conductance function $a_n:V_n \times V_n \to [0,\infty)$ satisfy the following:
\begin{enumerate}
  \item $a_n(u,v)=m_n(u)+m_n(v)$ whenever $u,v\in V_n$ such that $u \thicksim_{E_n} v$;
  \item For $u,v\in V_n$, we have $r_n(u,v)>0$ if and only if $u\thicksim_{E_n} v$; 
  \item For every $n\in \N$ and $u\in V_n$, we have
      $\frac{s}{K(1+s)}\mu_n(u)\leq m_n(u) \leq \frac{1}{1+s}\mu_n(u)$;
  \item For every $n\in \N$, the $(\mu_n,a_n,m_n)$-Cheeger constant is bounded below by $\frac{cs}{K}$.
\end{enumerate} 
\end{lem}

\begin{proof}
For every $n\in \N$, we consider the symmetric function $a_n: V_n \times V_n \to [0,\infty)$ defined by
\begin{equation*}
a_n(u,v):=
\begin{cases}
  ~m_n(u) + m_n(v), & \mbox{if~} u \thicksim_{E_n} v; \\
  ~0, & \mbox{otherwise,}
\end{cases}
\end{equation*}
for $u,v \in V_n$. Since $(V_n, E_n)$ is a connected finite graph by Remark~\ref{rem:measured-expanders-are-connected}, the associated stationary measure $\mu_n: V_n \to (0,\infty)$ is defined by
$$\mu_n(u):=\sum_{v \in V_n}a_n(u,v)\quad \text{for all $u\in V_n$. }$$
Then the formula $r_n(u,v):=a_n(u,v)/\mu_n(u)$ defines a reversible random walk on $V_n$ with the stationary measure $\mu_n$. Clearly, (1) and (2) hold by the above constructions. Moreover, (3) follows from the assumptions that $m_n(u) \thicksim_s m_n(v)$ for any edge $u\thicksim_{E_n} v$ and valencies are bounded by $K$ together with the connectedness of the graph $(V_n,E_n)$. 

As for (4), the assumption of $c$-measured expanders implies that for any $n\in \N$ and $A \subseteq V_n$ with $0<m_n(A) \leq \frac{1}{2}m_n(V_n)$ we have $m_n(\partial^{V_n} A) > c\cdot m_n(A)$. Hence,
\begin{eqnarray*}
a_n(\partial^{E_n} A) &=& \sum_{e \in \partial^{E_n} A} a_n(e) =  \sum_{u \in A, v\notin A, u\thicksim_{E_n} v} m_n(u)+m_n(v) \geq  \sum_{v \in \partial^{V_n} A}(1+s)m_n(v)\\
&=& (1+s)m_n(\partial^{V_n} A)> \frac{cs}{K}\mu_n(A),
\end{eqnarray*}
where we use (3) in the last inequality. Hence, we have verified (4) as desired.
\end{proof}

Now we are ready to prove the following $L^p$-Poincar\'{e} inequality for measured expanders with bounded measure ratios on adjacent vertices:
\begin{cor}\label{cor:Poincare for measured expanders}
Let $\{(V_n,E_n,m_n)\}_{n \in \N}$ be a sequence of $c$-measured expanders for some $c>0$ with valency uniformly bounded by $K\geq 1$. Assume that each $m_n$ has full support and there exists $s\in (0,1)$ such that $m_n(u) \thicksim_s m_n(v)$ for any edge $u \thicksim_{E_n} v$ and any $n\in \N$. 

Then for any $p\in [1,\infty)$, there exists a positive constant $c'$ only depending on $c,s,p,K$ such that for any $n\in \N$ and any map $f:V_n \to \C$, we have the following $L^p$-Poincar\'{e} inequality:
\begin{equation}\label{EQ13}
\sum_{u,v\in V_n:u\thicksim_{E_n} v}|f(u)-f(v)|^p (m_n(u) + m_n(v)) \geq c' \sum_{u,v\in V_n} |f(u)-f(v)|^p \frac{m_n(u)m_n(v)}{m_n(V_n)}.
\end{equation}
\end{cor}

\begin{proof}
First of all, we notice that for every $n\in \N$ there exists a reversible random walk $r_n$ on $V_n$ such that the associated stationary measure $\mu_n$ and the associated conductance function $a_n$ satisfy (1)-(4) in Lemma~\ref{lem:measured expanders to reversible ones}. Moreover, Lemma~\ref{lem:measured expanders to reversible ones}(2) tells us that $(V_n,E_n)$ is exactly the associated graph structure coming from the reversible random walk $r_n$ on $V_n$. Because of Lemma~\ref{lem:measured expanders to reversible ones}(4), Proposition~\ref{prop: Cheeger to Poincare} implies that for any $p\in [1,\infty)$  there exists a positive constant $c_p$ only depending on $c,s,p,K$ such that for any $n\in \N$ and any map $f: V_n \to \C$ we have that
\begin{equation*}
\sum_{u,v\in V_n:u\thicksim_{E_n} v}|f(u)-f(v)|^p a_n(u,v) \geq c_p \sum_{u,v\in V_n} |f(u)-f(v)|^p \frac{\mu_n(u)\mu_n(v)}{\mu_n(V_n)}.
\end{equation*}
By Lemma~\ref{lem:measured expanders to reversible ones}(1) and (3), we deduce the inequality (\ref{EQ13}) for $c':=\frac{s(1+s)c_p}{K}>0$.
\end{proof}

Finally, we show that the graph Laplacian (up to conjugacy) associated to measured expanders with bounded measure ratios has a spectral gap as expected.

\begin{prop}\label{prop: spetral gap for measured expanders}
Let $\{(V_n,E_n,m_n)\}_{n \in \N}$ be a sequence of measured expanders with uniformly bounded valency. Assume that each $m_n$ has full support and there exists $s\in (0,1)$ such that $m_n(u) \thicksim_s m_n(v)$ for any edge $u \thicksim_{E_n} v$ and any $n\in \N$. Let $(V,d)$ be a coarse disjoint union of $\{V_n\}_{n \in \N}$ with respect to the edge-path metrics, and $m$ be the direct sum measure of $\{m_n\}_{n \in \N}$ on $V$. Then for every $n\in \N$ there exists a reversible random walk on $V_n$ with a stationary measure $\mu_n:V_n\to (0,\infty)$ such that the following hold:
\begin{itemize}
\item[(1)] If $\Delta_n\in \B(\ell^2(V_n;\mu_n))$ is the graph Laplacian defined as in (\ref{EQ10}) and $\Lambda_n:=W_n^*\Delta_n W_n \in \B(\ell^2(V_n;m_n))$ where $W_n: \ell^2(V_n;m_n) \to \ell^2(V_n;\mu_n)$ is the ``set-theoretic identity'' operator, then the spectrum of $\Lambda:=\bigoplus_{n\in \N}\Lambda_n \in \B(\ell^2(V;m))$ is contained in $\{0\} \cup [\kappa,\infty)$ for some $\kappa>0$;
\item[(2)] Let $\mathfrak{S}_n \in \B(\ell^2(V_n;m_n))$ be the orthogonal projection onto the space of constant functions on $V_n$, and $\mathfrak{S}:=\bigoplus_{n\in \N} \mathfrak{S}_n$. Then $\mathfrak{S}=\chi_{\{0\}}(\Lambda)$ is the spectral projection; \item[(3)] $\sup\{d(u,v): \langle \Lambda\delta_v, \delta_u \rangle_{\ell^2(V;m)} \neq 0\}\leq 1$.
\end{itemize}
\end{prop}

\begin{proof}
From Lemma~\ref{lem:measured expanders to reversible ones}(3), we know that there exists a constant $M>0$ such that $||W_n||\leq M$ for every $n\in \N$. In particular, we have $\Lambda=\bigoplus_{n\in \N}\Lambda_n \in \B(\ell^2(V;m))$. To see the desired spectral gap of $\Lambda$, we apply Lemma~\ref{lem:measured expanders to reversible ones}(4) together with \cite[Lemma~3.19 and Proposition ~3.21]{measured_II}, which establish bounds between the spectral gap of $\Lambda$ and the best constant in a Poincar\'e inequality, and between the latter quantity and the Cheeger constant. This completes (1).

Since each $\Delta_n$ is positive, so are $\Lambda_n$ and $\Lambda$. As $\chi_{\{0\}}(\Lambda)=\bigoplus_{n\in \N} \chi_{\{0\}}(\Lambda_n)$, it suffices to show $\mathfrak{S}_n=\chi_{\{0\}}(\Lambda_n)$ for every $n\in\N$. Clearly, $\chi_{\{0\}}(\Lambda_n)$ is the orthogonal projection onto the kernel of $\Lambda_n$. Hence, we need to check that the kernel of $\Lambda_n$ exactly consists of constant functions on $V_n$. Since the identity operator $W_n$ is invertible, it follows that if $\eta \in \ell^2(V_n;m_n)$ then $\Lambda_n\eta=0$ if and only if $\Delta_nW_n\eta=0$. Since each $(V_n,E_n)$ is connected and $W_n^{-1}$ is also the identity operator, this is also equivalent to $\eta$ is a constant function on $V_n$. Hence, we have verified (2).

It is easily seen from (\ref{EQ10}) that $\sup\{d(u,v): \langle \Delta_n\delta_v, \delta_u \rangle_{\ell^2(V_n;\mu_n)} \neq 0\}\leq 1$ for every $n\in \N$. This fact clearly ensures (3) as desired. 
\end{proof}

\section{Proofs of main Theorems}
In previous sections, we introduced all necessary ingredients to prove our main theorems. Now we are in the position to prove Theorem \ref{thm: main result}, which is the foundation of this paper. As its proof is rather technical, we split it into two parts: we first consider the case of uniform Roe algebras and then the case of Roe algebras.

\subsection{The case of uniform Roe algebras} In this subsection, we prove the following theorem for block-rank-one projections in uniform Roe algebras:

\begin{thm}\label{thm:main result uniform case}
Let $\{(X_n,d_n)\}_{n\in \N}$ be a sequence of finite metric spaces with uniformly bounded geometry and $X$ be their coarse disjoint union. Let $P \in \B(\ell^2(X))$ be a block-rank-one projection with respect to $\{(X_n,d_n)\}_{n\in \N}$. Then $P$ is quasi-local \emph{if and only if} $P$ belongs to the uniform Roe algebra $C^*_u(X)$.
\end{thm}

\begin{proof}
The sufficiency holds trivially, so we only focus on the necessity. Suppose that $P=\bigoplus_{n\in \N} P_n$ is a quasi-local block-rank-one projection associated to the unit vectors $\{\xi_n\}_{n\in \N}$ in $\ell^2(X_n)$. Let $m_n$ be the associated probability measure on $X_n$ defined by $m_n(x):= |\xi_n(x)|^2$ for $x\in X_n$. By Lemma~\ref{cor: support quasi-local}, Lemma~\ref{lem: support Roe} and Proposition~\ref{prop: measured asymptotic expanders}, we may assume that each $m_n$ has full support on $X_n$ and $\{(X_n,d_n,m_n)\}_{n\in \N}$ is a sequence of measured asymptotic expanders. 

Then Corollary~\ref{cor:MAE to ME} guarantees that there exist $c>0$, a sequence $\{\alpha_k\}_{k \in \N}$ in $(0,1)$ with $\alpha_k\to 0$, a sequence $\{s_k\}_{k\in \N}$ in $(0,1)$, and a positive sequence $\{R_k\}_{k\in \N}$ such that for any $n, k \in \N$ there exist a finite graph $(V_{n,k},E_{n,k})$ with the edge-path metric $d_{n,k}$ and a $R_k$-Lipschitz injective map $i_{n,k}:V_{n,k} \to X_n$ satisfying the following:

  \begin{itemize}
 
    \item[(i)] the pullback measure $m_{n,k}:=i_{n,k}^*(m_n)$ on $V_{n,k}$ has full support and $m_{n,k}(V_{n,k}) \geq 1-\alpha_k$;
    \item[(ii)] for each $k\in \N$, $\{(V_{n,k},E_{n,k},m_{n,k})\}_{n\in \N}$ is a sequence of $c$-measured expanders with uniformly bounded valency;
    \item[(iii)] for any adjacent vertices $u\thicksim_{E_{n,k}} v$, we have $m_{n,k}(u) \thicksim_{s_k} m_{n,k}(v)$.
      \end{itemize}
  
Now we fix a $k\in \N$. Let $(V_k,d_k)$ be a coarse disjoint union of $\{(V_{n,k},d_{n,k})\}_{n\in \N}$ and $m_k$ be the direct sum measure of $\{m_{n,k}\}_{n\in \N}$ on $V_k$. For each $n\in \N$, let $\mathfrak{S}_{n,k} \in \B(\ell^2(V_{n,k}; m_{n,k}))$ be the orthogonal projection onto the space of constant functions on $V_{n,k}$ and $\mathfrak{S}_k:=\bigoplus_{n\in \N} \mathfrak{S}_{n,k} \in \B(\ell^2(V_k;m_k))$. Since $\{ m_{n,k}(v)^{-\frac{1}{2}}\delta_v\}_{v\in V_{n,k}}$ forms an orthonormal basis in $\ell^2( V_{n,k};m_{n,k})$ and $\{m_{n,k}(v)^{-\frac{1}{2}}\xi_n(i_{n,k}(v))\delta_{i_{n,k}(v)}\}_{v\in V_{n,k}}$ forms an orthonormal set in $\ell^2(X_n)$, we define a linear isometry $U_{n,k}: \ell^2(V_{n,k}; m_{n,k}) \rightarrow \ell^2(X_n)$ by the formula
$$m_{n,k}(v)^{-\frac{1}{2}}\delta_v \mapsto m_{n,k}(v)^{-\frac{1}{2}}\xi_n(i_{n,k}(v))\delta_{i_{n,k}(v)}, \text{ for $v\in V_{n,k}$.}$$
Thus, we have a linear isometry
$$U_k:=\bigoplus_{n\in \N} U_{n,k}: \ell^2(V_k; m_k) \rightarrow \ell^2(X).$$

From Proposition~\ref{prop: spetral gap for measured expanders}, there exists an operator $\Lambda_k \in \B(\ell^2(V_k;m_k))$ such that $\mathfrak{S}_k$ belongs to the unital $C^*$-subalgebra generated by $\Lambda_k$ in $\B(\ell^2(V_k;m_k))$ and 
\begin{align}\label{prop at most 1}
\sup\{d_k(u,v): \langle \Lambda_k\delta_v, \delta_u \rangle_{\ell^2(V_k;m_k)} \neq 0\}\leq 1. 
\end{align}
Hence, each $U_k \mathfrak{S}_k U_k^*$ belongs to the unital $C^*$-subalgebra generated by $U_k \Lambda_k U_k^*$ in $\B(\ell^2(X))$. On the other hand, the bounded operator $U_k \Lambda_k U_k^*$ has propagation at most $R_k$ as $i_{n,k}$ is $R_k$-Lipschitz together with (\ref{prop at most 1}). Consequently, both $U_k \Lambda_k U_k^*$ and $U_k \mathfrak{S}_k U_k^*$ belong to the uniform Roe algebra $C^*_u(X)$ for every $k\in \N$. In order to conclude $P\in C_u^*(X)$, it suffices to show that $U_k \mathfrak{S}_k U_k^*\to P$ as $k\to \infty$. 

Recall that $\mathfrak{S}_{n,k}$ in $\B(\ell^2(V_{n,k}; m_{n,k}))$ is the orthogonal projection onto the subspace spanned by the unit vector $\zeta_{n,k}\in \ell^2(V_{n,k};m_{n,k})$, where $\zeta_{n,k}(v)=\frac{1}{\sqrt{m_{n,k}(V_{n,k})}}=\frac{1}{\sqrt{m_n(\mathrm{Im} i_{n,k})}}$ for every $v\in V_{n,k}$. It follows easily that $U_{n,k}\mathfrak{S}_{n,k}U_{n,k}^*$ is the orthogonal projection onto the one-dimensional subspace spanned by the unit vector $\eta_{n,k}:=U_{n,k}\zeta_{n,k}\in \ell^2(X_n)$. A direct calculation gives us
\begin{equation*}
\eta_{n,k}(x)=
\begin{cases}
  ~\frac{\xi_n(x)}{\sqrt{m_n(\mathrm{Im} i_{n,k})}}\ , & x\in \mathrm{Im} i_{n,k}; \\
  ~0\ , & \mbox{otherwise}.
\end{cases}
\end{equation*}
Recall that $\xi_n$ is the unit vector in $\ell^2(X_n)$ corresponding to the rank-one projection $P_n$ in $\B(\ell^2(X_n))$.

As $1\geq m_n(\mathrm{Im} i_{n,k}) \geq 1-\alpha_k$ for all $n\in \N$, we have that
\begin{eqnarray*}
\|\eta_{n,k}-\xi_n\|^2 &=& \sum_{x\in \mathrm{Im} i_{n,k}} \big| \frac{\xi_n(x)}{\sqrt{m_n(\mathrm{Im} i_{n,k})}} - \xi_n(x) \big|^2 + \sum_{x \in X_n \setminus \mathrm{Im} i_{n,k}} |\xi_n(x)|^2\\
&=& \big( \frac{1}{\sqrt{m_n(\mathrm{Im} i_{n,k})}} -1 \big)^2\cdot m_n(\mathrm{Im} i_{n,k}) + m_n(X_n \setminus \mathrm{Im} i_{n,k})\\
&\leq & \big( \frac{1}{\sqrt{1-\alpha_k}} -1 \big)^2 + \alpha_k.
\end{eqnarray*}
Hence, it follows that
\begin{align*}
\|U_k \mathfrak{S}_k U_k^* - P\|^2 &=\sup_{n\in \N}\|U_{n,k}\mathfrak{S}_{n,k}U_{n,k}^*-P_n\|^2\\
     &\leq 4 \sup_{n\in \N}\|\eta_{n,k}-\xi_n\|^2 \\
     &\leq 4\left( \big( \frac{1}{\sqrt{1-\alpha_k}} -1 \big)^2 + \alpha_k\right)\rightarrow 0,\text{\ as $k\to \infty$.}
\end{align*}
Consequently, we deduce that $P$ belongs to $C_u^*(X)$ as desired.
\end{proof}

\subsection{The case of Roe algebras}
In this subsection, we complete Theorem~\ref{thm: main result}. In fact, we prove a stronger result (see Theorem~\ref{thm:main result Roe}). In order to state Theorem~\ref{thm:main result Roe}, we need the following notion:

\begin{defn}\label{defn:uniformization}
Let $(X,d)$ be a coarse disjoint union of a sequence of finite metric spaces $\{(X_n,d_n)\}_{n\in \N}$ and $\H_0$ be a Hilbert space. Let $P=\bigoplus_{n\in \N} P_n \in \B(\ell^2(X;\H_0))$ be a block-rank-one projection with respect to $\{X_n\}_{n\in \N}$. If $\xi_n$ is a unit vector in $\ell^2(X_n;\H_0)$ associated to the rank-one projection $P_n$, then the unit vector $\widetilde{\xi}_n \in \ell^2(X_n)$, defined by $\widetilde{\xi}_n(x):=\|\xi_n(x)\|$ for $x\in X_n$, is called the \emph{uniformization} of $\xi_n$. In this case, the rank-one projection $\widetilde{P}_n:=\langle \ \cdot\  , \widetilde{\xi}_n \rangle \widetilde{\xi}_n$ in $\B(\ell^2(X_n))$ is called the uniformization of $P_n$, and the block-rank-one projection $\widetilde{P}:=\bigoplus_{n\in \N}\widetilde{P}_n$ in $\B(\ell^2(X))$ is called the uniformization of $P$. 
\end{defn}

\begin{rem}\label{measure to uniformization}
It is clear that $P_n$ and $\widetilde{P}_n$ in Definition~\ref{defn:uniformization} have the same associated probability measure on $X_n$ given by $X_n \ni x\mapsto ||\xi_n(x)||^2=|\widetilde{\xi}_n(x)|^2$. Hence, $P$ is a ghost \emph{if and only if} $\widetilde{P}$ is a ghost by Lemma~\ref{lem: ghost}.
\end{rem}

The next lemma follows directly from Lemma~\ref{lem: quasi-locality of the projection} and the previous remark.

\begin{lem}\label{lem:equiv of QL}
Let $(X,d)$ be a coarse disjoint union of a sequence of finite metric spaces $\{(X_n,d_n)\}_{n\in \N}$, and $\H_0$ be a Hilbert space. Let  $P\in \B(\ell^2(X;\H_0))$ be a block-rank-one projection with respect to $\{X_n\}_{n\in \N}$, and let $\widetilde{P}\in \B(\ell^2(X))$ be the uniformization of $P$. Then $P$ is quasi-local \emph{if and only if} $\widetilde{P}$ is quasi-local.
\end{lem}
As expected, we also have the following counterpart to Lemma~\ref{lem:equiv of QL}:
\begin{lem}\label{lem:equiv of Roe}
Let $(X,d)$ be a coarse disjoint union of a sequence of finite metric spaces $\{(X_n,d_n)\}_{n\in \N}$ with uniformly bounded geometry, and $\H_0$ be an infinite-dimensional separable Hilbert space. Let $P\in \B(\ell^2(X;\H_0))$ be a block-rank-one projection with respect to $\{X_n\}_{n\in \N}$. If $\widetilde{P}\in \B(\ell^2(X))$ is the uniformization of $P$, then $P\in C^*(X)$ \emph{if and only if} $\widetilde{P} \in C^*_u(X)$.
\end{lem}

\begin{proof}
Let $\xi_n$ be a unit vector in $\ell^2(X_n;\H_0)$ associated to the rank-one projection $P_n$ for $n\in \N$, and we write $\xi_n = \sum_{x\in X_n} \delta_x \otimes \xi_n(x)$. Recall that its uniformization vector $\widetilde{\xi}_n \in \ell^2(X_n)$ is defined by $\widetilde{\xi}_n(x)=\|\xi_n(x)\|$ for $x\in X_n$. By Lemma~\ref{lem: support Roe}, we can assume that both $\xi_n$ and $\widetilde{\xi}_n$ have full support in $X_n$ for every $n\in \N$. 

For each $n\in \N$, we consider the linear isometry $W_n: \ell^2(X_n) \longrightarrow \ell^2(X_n;\H_0)$ defined by
\[
W_n(\delta_x) := \frac{\delta_x \otimes \xi_n(x)}{\|\xi_n(x)\|}, \text{\ $x\in X_n$.}
\]
Its adjoint operator $W_n^*: \ell^2(X_n;\H_0) \to \ell^2(X_n)$ is given by the formula
\[
W_n^*(\delta_x \otimes \eta)=\frac{\langle \eta, \xi_n(x) \rangle}{\|\xi_n(x)\|} \cdot \delta_x, \text{ for $x\in X_n$ and $\eta \in \H_0$.}
\]
It follows that
\begin{eqnarray*}
W_n(\widetilde{\xi}_n) &=& \sum_{x\in X_n} \|\xi_n(x)\| \cdot W_n(\delta_x) = \sum_{x\in X_n} \|\xi_n(x)\| \cdot \frac{\delta_x \otimes \xi_n(x)}{\|\xi_n(x)\|} = \sum_{x\in X_n} \delta_x \otimes \xi_n(x) = \xi_n;\\
W^*_n(\xi_n) &=& \sum_{x\in X_n} W_n^*(\delta_x \otimes \xi_n(x)) = \sum_{x\in X_n} \frac{\langle \xi_n(x), \xi_n(x) \rangle}{\|\xi_n(x)\|} \cdot \delta_x = \sum_{x\in X_n} \|\xi_n(x)\| \cdot \delta_x = \widetilde{\xi}_n.
\end{eqnarray*}
As $P_n=\langle \ \cdot\  , \xi_n \rangle \xi_n$ and $\widetilde{P}_n=\langle \ \cdot\  , \widetilde{\xi}_n \rangle \widetilde{\xi}_n$, we deduce that $P_n = W_n \circ \widetilde{P}_n \circ W_n^*$ and $\widetilde{P}_n = W_n^* \circ P_n \circ W_n$ for every $n\in \N$. Thus, $P = W \circ \widetilde{P} \circ W^*$ and $\widetilde{P} = W^* \circ P \circ W$, where $W:=\bigoplus_{n\in \N} W_n$ is the linear isometry from $\ell^2(X)$ to $\ell^2(X;\H_0)$. 

As $W$ and $W^*$ have propagation zero, we obtain a bounded linear map $C^*(X)\to C^*_u(X)$ given by $T\mapsto W^*\circ T\circ W$ for $T\in C^*(X)$. In particular, if $P\in C^*(X)$ then $\widetilde{P}=W^*\circ P\circ W\in C_u^*(X)$. Conversely, we want to show that $P\in C^*(X)$ if $\widetilde{P}\in C_u^*(X)$. Let $T$ be any bounded operator on $\ell^2(X)$ with finite propagation. If $x\in X_k$ and $y\in X_l$, then we have
$$
(W\circ T\circ W^*)_{x,y}\xi=\frac{\langle \xi  , \xi_l(y) \rangle \cdot T_{x,y}}{||\xi_l(y)||\cdot||\xi_k(x)||}\cdot\xi_k(x) \text{ for $\xi \in \H_0$.}
$$
Thus, $W\circ T\circ W^*$ has finite propagation and is locally compact as $(W\circ T\circ W^*)_{x,y}$ is a rank-one operator on $\H_0$ for every $x,y\in X$. As a consequence, we obtain an injective $*$-homomorphism $C_u^*(X)\to C^*(X)$ given by $T \mapsto W\circ T\circ W^*$ so that  $P = W \circ \widetilde{P} \circ W^*\in C^*(X)$. Hence, we complete the proof.
\end{proof}

Combining Theorem~\ref{thm:main result uniform case} with Lemma~\ref{lem:equiv of QL} and Lemma~\ref{lem:equiv of Roe}, we obtain our main theorem in this subsection and finish the proof of Theorem~\ref{thm: main result}:
\begin{thm}\label{thm:main result Roe}
Let $(X,d)$ be a coarse disjoint union of a sequence of finite metric spaces $\{(X_n,d_n)\}_{n\in \N}$ with uniformly bounded geometry, and $\H_0$ be an infinite-dimensional separable Hilbert space. Let  $P\in \B(\ell^2(X;\H_0))$ be a block-rank-one projection with respect to $\{X_n\}_{n\in \N}$. If $\widetilde{P}\in \B(\ell^2(X))$ is the uniformization of $P$, then the following are equivalent:
\begin{enumerate}
  \item $P\in C^*(X)$;
  \item $P$ is quasi-local;
  \item $\widetilde{P} \in C^*_u(X)$;
  \item $\widetilde{P}$ is quasi-local. 
\end{enumerate}
\end{thm}

\begin{rem}\label{lift by rank-1 proj}
If $P\in C_u^*(X)$ is a block-rank-one projection with respect to a coarse disjoint union $X=\bigsqcup_{n\in \N}X_n$, then so is $P\otimes e\in C^*(X)$ for any fixed rank-one projection $e\in \K(\H_0)$. In this case, the uniformization of $P\otimes e$ is $P$ itself. Consequently,  Theorem~\ref{thm:main result uniform case} can be deduced from Theorem~\ref{thm:main result Roe}.
\end{rem}

Let us finish this subsection with some consequences for the coarse Baum-Connes conjecture. First of all, we improve Proposition~\ref{prop:general case of projections in Roe} by providing the following measured version of \cite[Proposition~6.4]{structure}:
\begin{prop}\label{prop: weak embedding and ghost projection.finite to 1 uniform case}
Let $Y$ be a metric space of bounded geometry. Assume that there exists a sequence of measured asymptotic expanders $\{(X_n,d_n,m_n)\}_{n\in \N}$ of uniformly bounded geometry which admits a measured weak embedding $\{f_n\colon (X_n,d_n,m_n) \to Y\}_{n\in \N}$ such that for every $y\in Y$ the preimage $f_n^{-1}(y)$ is empty for all but finitely many $n\in \N$. Then there exist a sparse subspace $Y' \subseteq Y$ and a block-rank-one ghost projection in $C^*_u(Y')$. 

If we additionally assume that the sequence $\{f_n\}_{n\in \N}$ is \emph{uniformly} finite-to-one and each $m_n$ is the counting measure, then the averaging projection of $Y'$ is a ghost in $C^*_u(Y')$.
\end{prop}

\begin{proof}
Without loss of generality, we can assume that each $m_n$ is a probability measure. Let $X$ be a coarse disjoint union of $\{(X_n,d_n)\}_{n\in \N}$, and $f:=\bigsqcup_{n\in \N} f_n: X\to Y$. As the sequence $\{f_n\}_{n\in \N}$ is coarse, we can assume that $f$ is coarse as well by appropriately choosing the metric on $X$. From the assumption on the preimages of $\{f_n\}_{n\in \N}$, we know that $f$ is finite-to-one.
 
By Proposition~\ref{prop: measured asymptotic expanders} and Theorem~\ref{thm:main result uniform case}, the associated block-rank-one projection $P:=\bigoplus_{n\in \N} P_n$ with respect to $\{(X_n,d_n,m_n)\}_{n\in \N}$ belongs to $C_u^*(X)$, where each $P_n \in \B(\ell^2(X_n))$ is the rank-one projection onto the subspace spanned by the unit vector $X_n\ni x\mapsto \sqrt{m_n(x)}$ in $\ell^2(X_n)$. We fix an infinite-dimensional separable Hilbert space $\H_0$ and a rank-one projection $e \in \K(\H_0)$, then $P \otimes e \in C^*(X)$ and its uniformization is $P$ itself by Remark~\ref{lift by rank-1 proj}. Moreover, Remark~\ref{measure to uniformization} tells us that each $m_n$ is also the associated measure to $P_n \otimes e$. Hence, it follows from Proposition~\ref{prop:general case of projections in Roe} that there exist a sparse subspace $Y'=\bigsqcup_{n\in \N} Y_n$ in $Y$ where $Y_n=f_n(X_n)$, and a block-rank-one ghost projection $Q \in C^*(Y')$. By Theorem~\ref{thm:main result Roe} and Remark~\ref{measure to uniformization}, we conclude that its uniformization $\widetilde{Q}$ is a block-rank-one ghost projection in $C^*_u(Y')$ as desired.

Now we additionally assume that the sequence $\{f_n\}_{n\in \N}$ is uniformly finite-to-one and each $m_n$ is the probability counting measure on $X_n$. As $\widetilde{Q}$ is a block-rank-one ghost projection in $C^*_u(Y')$, it follows from Proposition~\ref{prop: measured asymptotic expanders} and Lemma~\ref{lem: ghost} that $\{(Y_n,m_n')\}_{n\in \N}$ forms a ghostly sequence of measured asymptotic expanders, where $m_n'$ is the probability measure on $Y_n$ associated to $\widetilde{Q}_n$. From the formula (\ref{formula of measure on Yn}) in the proof of Proposition~\ref{prop:general case of projections in Roe} we have that
$$
m_n'(y)=||\widetilde{Q}_n\delta_y||^2=|(\widetilde{Q}_n)_{y,y}|=m_n(f_n^{-1}(y))=\frac{|f_n^{-1}(y)|}{|X_n|}, \text{\ for $y\in Y_n$}.
$$

Since $\{f_n\}_{n\in \N}$ is uniformly finite-to-one and each $f_n$ is surjective onto $Y_n$, there exists a constant $s\geq 1$ such that for any $n\in \N$ and any $y_1,y_2\in Y_n$, we have 
\[
\frac{1}{s} m'_n(y_2) \leq m'_n(y_1) \leq s m'_n(y_2).
\]
As a consequence of \cite[Definition~6.4 and Lemma~6.5]{dypartI}, we conclude that $\{Y_n\}_{n\in \N}$ is a sequence of asymptotic expanders with $|Y_n|\to \infty$ as $n\to \infty$. So the averaging projection of $\{Y_n\}_{n\in \N}$ is a ghost in $C^*_u(Y')$ by Theorem \ref{thm:main result uniform case} or \cite[Theorem~C]{structure}.
\end{proof}

Note that in the previous proposition the condition ``for every $y\in Y$ the preimage $f_n^{-1}(y)$ is empty for all but finitely many $n\in \N$'' is redundant if $Y$ is sparse. This follows from Corollary~\ref{cor:measured weak embedding} and the following general fact:

\begin{lem}\label{condition the preimage}
Let $\{(V_n,E_n,m_n)\}_{n\in \N}$ be a sequence of finite measured graphs and $Y=\bigsqcup_{k\in \N}Y_k$ be a sparse space. If $\{f_n: (V_n,E_n,m_n) \to Y\}_{n\in \N}$ is a measured weak embedding, then for every $y\in Y$ the preimages $f_n^{-1}(y)\neq \emptyset$ for only finitely many $n$'s.
\end{lem}
\begin{proof}
If this were not the case, then there would exist $y_0\in Y$ such that $f_{n_i}^{-1}(y_0)\neq \emptyset$ for all $i\in \N$. As all $(V_n,E_n)$ are connected graphs and all $f_n$ are $L$-Lipschitz for some $L>0$ by Remark~\ref{graph mwe}, it follows that $f_{n_i}(V_{n_i})\subseteq E:=\bigsqcup_{k=1}^M Y_k$ for some natural number $M$ (independent of $i$). Since $E$ is a non-empty finite set and $f_{n_i}^{-1}(E)=V_{n_i}$ for all $i\in \N$, it follows that $m_{n_i}(V_{n_i})=\sum_{y\in E}m_{n_i}(f_{n_i}^{-1}(y))$ for every $i\in \N$. In particular, for each $i\in \N$ there exists $y_i\in E$ such that $\frac{m_{n_i}(f_{n_i}^{-1}(y_i))}{m_{n_i}(V_{n_i})}\geq \frac{1}{|E|}>0$, which contradicts with the assumption that $\{f_n\}_{n\in \N}$ is a measured weak embedding.
\end{proof}

The following theorem is a consequence of Corollary~\ref{cor:measured weak embedding}, Proposition~\ref{prop: weak embedding and ghost projection.finite to 1 uniform case}  and Lemma~\ref{condition the preimage} with an identical proof of \cite[Theorem~6.7]{structure}.
\begin{thm}\label{thm: existence of ghost proj in Roe}
Let $\{X_n\}_{n\in \N}$ be a sequence of finite metric spaces of uniformly bounded geometry and $X$ be their coarse disjoint union. If $X$ admits a fibred coarse embedding into a Hilbert space\footnote{See \cite[Definition~2.1]{CWY13} for the definition of fibred coarse embedding into Hilbert spaces.}, and there exists a sequence of measured asymptotic expanders of uniformly bounded geometry which admits a measured weak embedding into $X$, then the following statements hold:
\begin{itemize}
\item[(1)] The coarse Baum--Connes assembly map for $X$ is injective but non-surjective.
\item[(2)] The induced map $\iota_*\colon K_*(\K)\rightarrow K_*(I_G)$ is injective but non-surjective, where $\iota\colon\K\hookrightarrow I_G$ is the inclusion of the compact ideal $\K$ into the ghost ideal $I_G$ of the Roe algebra $C^*(X)$.
\item[(3)] The induced map $\pi_*\colon K_*(C^*_{max}(X))\rightarrow K_*(C^*(X))$ is injective but non-surjective, where $\pi\colon C^*_{max}(X)\twoheadrightarrow C^*(X)$ is the canonical surjection from the maximal Roe algebra onto the Roe algebra. In particular, $\pi$ is not injective.
\end{itemize}
\end{thm}

In view of Lemma~\ref{lem:measured.weak.embedding}, we deduce the following corollary of Theorem~\ref{thm: existence of ghost proj in Roe}:
\begin{cor}\label{cbc+gmae}
Let $\{X_n\}_{n\in \N}$ be a sequence of finite metric spaces of uniformly bounded geometry and $X$ be their coarse disjoint union. If $X$ admits a fibred coarse embedding into a Hilbert space, and there exists a sequence of ghostly measured asymptotic expanders of uniformly bounded geometry which admits a coarse embedding into $X$, then $X$ violates the coarse Baum-Connes conjecture.
\end{cor}

\subsection{A geometric criterion for the rigidity problem}\label{The geometric condition and its characterisations}
Now we provide a geometric condition for which the rigidity problem holds. Let us start with the following characterisations for the analytic conditions introduced in Definition \ref{our geometric condition}: 
\begin{cor}\label{cor:geometric condition characterisation}
Let $X$ be a metric space with bounded geometry. Then the following properties are equivalent: 
\begin{enumerate}
\item all sparse subspaces of $X$ contain no block-rank-one ghost projections in their Roe algebras;
\item all sparse subspaces of $X$ contain no block-rank-one ghost projections in their uniform Roe algebras;
\item $X$ contains no sparse subspaces consisting of ghostly measured asymptotic expanders.
\item $X$ coarsely contains no sparse spaces consisting of ghostly measured asymptotic expanders with uniformly bounded geometry. 
\end{enumerate}
In particular, all of these properties are coarsely invariant among metric spaces with bounded geometry.
\end{cor}

\begin{proof}
``$(1) \Rightarrow (2)$'' follows from Remark~\ref{measure to uniformization} and Remark~\ref{lift by rank-1 proj}. While ``$(2) \Rightarrow (1)$'' follows from Remark~\ref{measure to uniformization} and ``$(1) \Rightarrow (3)$'' in Theorem~\ref{thm:main result Roe}. The equivalence between $(2)$ and $(3)$ follows directly from Lemma~\ref{lem: ghost}, Proposition~\ref{prop: measured asymptotic expanders} and Theorem~\ref{thm:main result uniform case}.

Since every isometric embedding is a coarse embedding, we clearly have ``$(4) \Rightarrow (3)$''. To see ``$(3) \Rightarrow (4)$'', we assume that $X$ coarsely contains a sparse space $X'$ consisting of ghostly measured asymptotic expanders with uniformly bounded geometry. If $f:X'\to X$ is the coarse embedding, then $f$ must be finite-to-one because $X'$ is sparse. As a consequence of Lemma~\ref{lem:measured.weak.embedding} and Proposition~\ref{prop: weak embedding and ghost projection.finite to 1 uniform case}, there exist a sparse subspace $Z$ of $X$ and a block-rank-one ghost projection in $C_u^*(Z)$. Hence, $Z$ consists of ghostly measured asymptotic expanders by Lemma~\ref{lem: ghost} and Proposition~\ref{prop: measured asymptotic expanders}. 

Finally, the last statement can be deduced from Proposition~\ref{prop:coarse invariant} together with the fact that every coarse embedding from a metric space with bounded geometry is uniformly finite-to-one.
\end{proof}

Condition (3) in Corollary~\ref{cor:geometric condition characterisation} is the one we are most interested in, because it is formally the weakest geometric criterion which guarantees  rigidity. More precisely, we obtain the following rigidity result (Theorem~\ref{introthm:rigidity.geometric.condition}) by combining Proposition~\ref{prop:rigidity.Roe} with Corollary~\ref{cor:geometric condition characterisation}. 
\begin{thm}\label{thm:rigidity.geometric.condition}
Let $X$ and $Y$ be metric spaces with bounded geometry. Assume that either $X$ or $Y$ contains no sparse subspaces consisting of ghostly measured asymptotic expanders. Then the following are equivalent:
\begin{enumerate}
  \item $X$ is coarsely equivalent to $Y$;
  \item $C^*_u(X)$ is Morita equivalent \footnote{By \cite[Theorem~1.2]{BGR}, $C^*_u(X)$ and $C^*_u(Y)$ are Morita equivalent if and only if they are stably $\ast$-isomorphic.} to $C^*_u(Y)$;
  \item $C^*_s(X)$ is $\ast$-isomorphic to $C^*_s(Y)$;
  \item $UC^*(X)$ is $\ast$-isomorphic to $UC^*(Y)$;
  \item $C^*(X)$ is $\ast$-isomorphic to $C^*(Y)$.
\end{enumerate}
\end{thm}

\begin{rem}
Very recently, it has been shown in \cite{BBFKVW} that (1) to (4) in Theorem~\ref{thm:rigidity.geometric.condition} are all equivalent for general metric spaces with bounded geometry. At this moment, we believe that Theorem~\ref{thm:rigidity.geometric.condition} is still the best-known result on the equivalence of (5) with the other items. 
\end{rem}

\section{Rigid metric spaces}\label{sec:examples}
In this final section, we will apply Theorem~\ref{thm:rigidity.geometric.condition} to obtain new examples of rigid spaces which are not obviously covered by previously existing results.

Let us start with the following result, which asserts that measured asymptotic expanders are obstructions to measured weak embeddings into $L^p$-spaces.

\begin{prop}\label{prop: non CE}
Any sequence of measured asymptotic expanders with uniformly bounded geometry cannot be measured weakly embedded into any $L^p$-space for $p \in [1,\infty)$. 
\end{prop}

\begin{proof}
We assume that there is a sequence of measured asymptotic expanders $\{(X_n,d_n,m_n)\}_{n\in \N}$ with uniformly bounded geometry which could be measured weakly embedded into $L^p(Z;\nu)$ for some measure space $(Z,\nu)$ and $p \in [1,\infty)$. By Corollary~\ref{cor:MAE to ME}, we can assume that $\{(X_n,d_n,m_n)\}_{n\in \N}$ is a sequence of measured expander graphs with valency uniformly bounded by $K\geq 1$ such that all $m_n$ have full support with uniformly bounded measure ratio $s\in (0,1)$ on adjacent vertices.

By Corollary~\ref{cor:Poincare for measured expanders}, there exists a positive constant $c$ such that for any $n\in \N$ and any map $g:X_n \to \C$, we have the following $L^p$-Poincar\'{e} inequality:
\begin{equation}\label{EQ8}
\sum_{u,v\in X_n:u\thicksim_{} v}|g(u)-g(v)|^p (m_n(u) + m_n(v)) \geq c \sum_{u,v\in X_n} |g(u)-g(v)|^p \frac{m_n(u)m_n(v)}{m_n(X_n)}.
\end{equation}
If $\{f_n: (X_n,d_n,m_n) \to L^p(Z,\nu)\}_{n\in \N}$ is a measured weak embedding, then all maps $f_n$ can be chosen to be $L$-Lipschitz for some $L>0$ by Remark~\ref{graph mwe}. Integrating the inequality (\ref{EQ8}) over $(Z,\nu)$, we obtain that
$$\sum_{u,v\in X_n} \|f_n(u)-f_n(v)\|^p \frac{m_n(u)m_n(v)}{m_n(X_n)} \leq \frac{L^p}{c} \sum_{u,v\in X_n:u\thicksim_{} v}(m_n(u) + m_n(v))\leq \frac{(1+s)KL^p}{sc} m_n(X_n).$$
As $\sum_{u\in X_n}\frac{m_n(u)}{m_n(X_n)}=1$, using the pigeonhole principle we deduce that there exists at least one $u\in X_n$ such that
$$\sum_{v\in X_n} \|f_n(u)-f_n(v)\|^p \frac{m_n(v)}{m_n(X_n)} \leq \frac{(1+s)KL^p}{sc}.$$

If we denote $M:=\frac{(1+s)KL^p}{sc}$, then the set $\{v\in X_n:\|f_n(u)-f_n(v)\|^p\leq 2 M\}$ has measure greater than $\frac{1}{2}m_n(X_n)$ for every $n\in \N$. Equivalently, we have
$$
\frac{m_n(f_n^{-1}(B(f_n(u),(2M)^{1/p})))}{m_n(X_n)}>\frac{1}{2}\text{\ for every $n\in \N$}.
$$
This contradicts with the assumption that $\{f_n\}_{n\in \N}$ is a measured weak embedding.
\end{proof}

By Lemma~\ref{lem:measured.weak.embedding} and Proposition~\ref{prop: non CE}, we obtain the following result which extends the $L^p$-case of \cite[Theorem 4.2]{structure}.
\begin{cor}\label{cor: non CE}
A coarse disjoint union of a sequence of ghostly measured asymptotic expanders with uniformly bounded geometry cannot coarsely embed into any $L^p$-space for $p \in [1,\infty)$. 
\end{cor}

As a consequence of Theorem~\ref{thm:rigidity.geometric.condition}, rigidity holds for metric spaces which coarsely embed into some $L^p$-space for $p \in [1,\infty)$ as stated in Corollary \ref{introcor:lp}.

We finish this section by extending the main results in \cite{AT15, delabie2018box} to the measured asymptotic case. More precisely, we show that there exist box spaces that do not coarsely embed into any $L^p$-space, but do not measured weakly contain any measured asymptotic expanders. Consequently, we obtain that the rigidity problem holds for such box spaces by Lemma~\ref{lem:measured.weak.embedding} and Theorem~\ref{thm:rigidity.geometric.condition}.

To this end, we need the following key ingredient, which is a measured asymptotic analogue of \cite[Proposition 2]{AT15}:
\begin{prop}\label{prop:AT15 analogue}
Let $\{G_n\}_{n\in \N}$ be a sequence of finitely generated groups with generating sets $\{S_n\}_{n\in \N}$ such that $|S_n| \leq k$ for all $n\in \N$. We assume that for each $n\in \N$, we have a short exact sequence
\[
1 \rightarrow N_n \rightarrow G_n \rightarrow Q_n \rightarrow 1
\]
such that:
\begin{itemize}
  \item The sequence $(N_n)_{n\in \N}$ equipped with the induced metric coarsely embeds into a Hilbert space; 
  \item The sequence $(Q_n)_{n\in \N}$ equipped with the word metric associated to the projection $T_n$ of $S_n$ coarsely embeds into a Hilbert space.
\end{itemize}
Then the coarse disjoint union $G:=\bigsqcup_{n\in\N} G_n$ does not measured weakly contain any measured asymptotic expanders with uniformly bounded geometry. 
\end{prop}

Before we prove Proposition~\ref{prop:AT15 analogue}, we need the following lemma, which is the measured analogue of \cite[Lemma~2.1 and Corollary~2.2]{AT15}:

\begin{lem}\label{Step II}
Let $p\in [1,\infty)$, $\alpha\in (0,1]$ and $c>0$. If $\{(V_n,E_n,m_n)\}_{n\in \N}$ is a sequence of $c$-measured expander graphs with valency uniformly bounded by $K\geq 1$ such that all $m_n$ have full support with uniformly bounded measure ratio $s\in (0,1)$ on adjacent vertices, then there exists $D>0$ depending only on $p,\alpha,c,K$ and $s$ such that for any $n\in \N$, any $A\subseteq V_n$ with $m_n(A) \geq \alpha \cdot m_n(V_n)$ and any $1$-Lipschitz map $f$ from $A$ to an $L^p$-space, we have
\begin{equation}\label{EQ12}
\sum_{u,v\in A} \|f(u)-f(v)\|^p \frac{m_n(u)m_n(v)}{m_n(A)^2} \leq D.
\end{equation}
In particular, there exists $x_0 \in A$ such that the set $\{v\in A:\|f(x_0)-f(v)\|\leq (2D)^{\frac{1}{p}}\}$ has measure greater than $m_n(A)/2$ for every $n\in \N$.
\end{lem}
\begin{proof}
We will follow the same scheme as the proof of \cite[Lemma~2.1]{AT15} with minor modifications. There are two cases: $\alpha=1$ and $\alpha\in (0,1)$.

If $\alpha=1$, then $A=V_n$ because $m_n$ has full support. Let $f$ be any $1$-Lipschitz map from $V_n$ to an $L^p$-space. By Corollary~5.3, there exists a positive constant $c'$ only depending on $c,s,p, K$ such that
$$
\sum_{u,v\in V_n} \|f(u)-f(v)\|^p \frac{m_n(u)m_n(v)}{m_n(V_n)^2} \leq \frac{1}{c'\cdot m_n(V_n)} \sum_{u,v\in V_n:u\thicksim_{E_n} v}(m_n(u) + m_n(v))\leq \frac{(1+s)K}{sc'}.
$$
In this case, we simply choose $D:=\frac{(1+s)K}{sc'}$ so that (\ref{EQ12}) holds.

If $\alpha\in (0,1)$, then $1-\alpha\in (0,1)$. By Lemma~\ref{lem:annoying 1/2 bdd expander}, there exists a $c_{1-\alpha}>0$ depending only on $c$ and $\alpha$ such that
\[
h^n_\alpha:=\min\big\{\frac{m_n(\partial^{V_n} B)}{m_n(B)}: B\subseteq V_n \text{ with } 0<m_n(B) \leq (1-\alpha)\cdot m_n(V_n)\big\}>c_{1-\alpha}
\]
for all $n\in \N$.

Given $A\subseteq V_n$ with $m_n(A) \geq \alpha \cdot m_n(V_n)$. For every $i\in \N$, set $U_i:= \Nd_i(A)^c$ and $W_i:= U_i \setminus U_{i+1}$. As $(\partial^{V_n} U_{i+1}) \sqcup U_{i+1} \subseteq U_i$, and $m_n(U_{i+1})\leq (1-\alpha)\cdot m_n(V_n)$ for every $i\in \N$, it follows that
\[
m_n(U_{i}) \geq m_n(\partial^{V_n} U_{i+1}) + m_n(U_{i+1}) \geq (1+h^n_\alpha)\cdot m_n(U_{i+1}).
\]
Thus, for every $i\in \N$
\[
m_n(U_{i+1}) \leq \frac{m_n(U_1)}{(1+h^n_\alpha)^i} < \frac{m_n(V_n)}{(1+c_{1-\alpha})^i}.
\]

Now we extend $f$ from $A$ to all of $X$ as follows: for any $x\in A^c$, we choose a point $a_x\in A$ such that $d_n(x,A)=d_n(x,a_x)$ and we set $f(x):=f(a_x)$. If $x,y\in A^c$ with $x \thicksim_{E_n} y$, then $d_n(a_x,a_y)\leq d_n(x,A)+d_n(y,A)+1$. Thus, we deduce that $d_n(a_x,a_y) \leq 2i+4$ for every $x\in W_i$ and every $y \in V_n$ with $x \thicksim_{E_n} y$. As $f$ is $1$-Lipschitz on $A$, then $\|f(x)-f(y)\|\leq 2i+4$ for every $x\in W_i$ with $x \thicksim_{E_n} y$. As $A^c=\bigsqcup_{i\in \N}W_i$, we have that
\begin{align*}
&\sum_{u\in A^c}\sum_{v\in V_n:\ u\thicksim_{E_n} v} \| f(u) - f(v)\|^p (m_n(u) + m_n(v))\\
&\leq \sum_{i\in \N}\sum_{u\in W_i}\sum_{v\in V_n:\ u\thicksim_{E_n} v} (2i+4)^p (m_n(u) + m_n(v))\\
&\leq  \sum_{i\in \N} (2i+4)^p \frac{(1+s)K}{s} m_n(W_i)\\
&<\frac{(1+s)K}{s} m_n(V_n) \sum_{i\in \N}\frac{(2i+4)^p}{(1+c_{1-\alpha})^{i-1}},
\end{align*}
where $\sum_{i\in \N}\frac{(2i+4)^p}{(1+c_{1-\alpha})^{i-1}}$ is clearly convergent and its limit denotes by $\theta(p,c,\alpha)$. Moreover, we have $\|f(x)-f(y)\|\leq 2$ if $x\in A$ and $x \thicksim_{E_n} y$. Hence, it follows that
\begin{align*}
\sum_{u\in A}\sum_{v\in V_n:\ u\thicksim_{E_n} v} \| f(u) - f(v)\|^p (m_n(u) + m_n(v))\leq 2^p\frac{(1+s)K}{s}m_n(V_n).
\end{align*}
Combining them together, we conclude that
\begin{align*} 
&\sum_{u,v\in V_n:u\thicksim_{E_n} v} \| f(u) - f(v)\|^p (m_n(u) + m_n(v))\\
&<2^p\frac{(1+s)K}{s}m_n(V_n)+\frac{(1+s)K}{s}\cdot\theta(p,c,\alpha)\cdot m_n(V_n)\\
&=\frac{(1+s)K}{s}(2^p+\theta(p,c,\alpha))\cdot m_n(V_n).
\end{align*}

Now Corollary~\ref{cor:Poincare for measured expanders} provides us a positive constant $c'$ depending only on $p, K, s$ and $c$ so that 
\begin{align*}
&\sum_{u,v\in V_n} \|f(u)-f(v)\|^p \frac{m_n(u)m_n(v)}{m_n(V_n)^2} \\
&\leq \frac{1}{c'\cdot m_n(V_n)} \sum_{u,v\in V_n:u\thicksim_{E_n} v}\|f(u)-f(v)\|^p(m_n(u) + m_n(v))\\
& < \frac{(1+s)K}{c's}(2^p+\theta(p,c,\alpha)).
\end{align*}
As $m_n(A) \geq \alpha \cdot m_n(V_n)$, we conclude that (\ref{EQ12}) holds for $D:=\frac{(1+s)K}{\alpha^2c's}  \big(2^p+\theta(p, c, \alpha)\big)$.
As for the last statement, we simply apply the pigeonhole principle to (\ref{EQ12}) as in the proof of Proposition~\ref{prop: non CE}.
\end{proof}

\begin{proof}[Proof of Proposition~\ref{prop:AT15 analogue}]
We assume that there is a sequence of measured asymptotic expanders with uniformly bounded geometry which could be measured weakly embedded into $G$. From Corollary~\ref{cor:MAE to ME}, we obtain a sequence of $c$-measured expander graphs $\{(V_n,E_n,m_n)\}_{n\in \N}$ with valency uniformly bounded by $K\geq 1$ such that all $m_n$ have full support with uniformly bounded measure ratio $s\in (0,1)$ on adjacent vertices such that the sequence $\{(V_n,E_n,m_n)\}_{n\in \N}$ measured weakly embeds into $G$. If $\{g_n: (V_n,E_n,m_n) \to G\}_{n\in \N}$ is such a measured weak embedding, then for every $y\in G$ the preimages $g_n^{-1}(y)\neq \emptyset$ for only finitely many $n$'s by Lemma~\ref{condition the preimage}. Since all $(V_n, E_n)$ are connected graphs and all $g_n$ are $M$-Lipschitz for some $M>0$, we may remove finitely many $V_n$'s if necessary so that every $V_n$ in the remaining sequence is mapped into a single piece $G_{k_n}$ for some $k_n\in \N$. Using Lemma~\ref{condition the preimage} again, we may (after taking a subsequence if necessary) assume that each $g_n$ maps $V_n$ into $G_n$ for all $n\in \N$.

We reach a contradiction with the fact that $\{g_n\}_{n\in \N}$ is a measured weak embedding by following the proof of \cite[Proposition~2]{AT15} step by step, except we use Lemma~\ref{Step II} instead of \cite[Corollary~2.2]{AT15}, and replace the counting measure by $m_n$.
\end{proof}

Now we are able to strengthen \cite[Theorem~1]{AT15} as follows:
\begin{thm}\label{AT example}
There exists a box space $X$ of a finitely generated residually finite group such that $X$ does not coarsely embed into any $L^p$-space for $1\leq p<\infty$\footnote{The box space $X$ also does not coarsely embed into any \emph{uniformly curved} Banach space (see \cite[Definition~6.1]{AT15} for the definition).}, yet does not measured weakly contain any measured asymptotic expanders with uniformly bounded geometry.
\end{thm}

\begin{proof}
The proof is identical to the proof of \cite[Theorem~7.1]{AT15} except that we use Proposition~\ref{prop:AT15 analogue} instead of \cite[Proposition~2]{AT15} at the very end of the proof. For this reason, we shall not repeat the argument.
\end{proof}

In fact, the graphs constituting the box space $X$ in Theorem~\ref{AT example} can be additionally required to have unbounded  \emph{girth} (\emph{i.e.}, the length of the shortest cycle). For this, we need the following stronger version of \cite[Proposition 2.4]{delabie2018box}, whose proof is the same except that we use Proposition~\ref{prop:AT15 analogue} instead of \cite[Proposition 2]{AT15} therein. To avoid too much word repetition, we omit its proof here.

\begin{cor}\label{new delabbie-Khukhro}
Let $G$ be a finitely generated, residually finite group and let $\{N_n\}_{n\in \N}$ be a sequence of nested finite index normal subgroups of $G$ with trivial intersection. If $\{M_n\}_{n\in \N}$ is another sequence of finite index normal subgroups of $G$ with $N_n > M_n$ for all $n\in \N$ and the box space $\bigsqcup_{n\in \N} G/N_n$ coarsely embeds into a Hilbert space, then the box space $\bigsqcup_{n\in \N} G/M_n$ does not measured weakly contain any measured asymptotic expanders with uniformly bounded geometry.
\end{cor}

Now we are ready to generalise \cite[Theorem~4.5]{delabie2018box} in the following way:
\begin{thm}\label{box space of F3}
There exists a box space $X$ of the free group $F_3$ such that $X$ does not coarsely embed into any $L^p$-space for $1\leq p<\infty$, yet does not measured weakly contain any measured asymptotic expanders with uniformly bounded geometry.
\end{thm}

\begin{proof}
We refer to the proof of \cite[Theorem~4.5]{delabie2018box} for details. In the proof, we have to apply our Corollary~\ref{new delabbie-Khukhro} instead of \cite[Proposition 2.4]{delabie2018box} herein in order to obtain the second part of the statement, while the first part follows from the last remark in \cite[Section~5]{delabie2018box}\footnote{The authors claimed in this remark that their box space does not coarsely embed into $\ell^p$, but it is straightforward to check that it actually cannot coarsely embed into any $L^p$-space for $1\leq p<\infty$ by the complex interpolation method as in the proof of \cite[Proposition~4]{AT15}).}.
\end{proof}

\bibliographystyle{plain}
\bibliography{expanderish_graph_measured}
\end{document}